\documentclass[11pt]{article}
\usepackage{amssymb,amsmath,float}
\usepackage{amsthm}  %Remove for siamart

\usepackage[total={6.5in,8.75in}, top=1.2in, left=0.9in, includefoot]{geometry}
\usepackage{graphicx}
\usepackage{tikz}
\usepackage{enumitem}
 \usepackage[T1]{fontenc}    

\usepackage{url}
\usepackage[pdftex, bookmarks, colorlinks, citecolor =blue, pdfencoding=auto, psdextra]{hyperref}  %allows for $$ in hyperlinks

\usepackage{xcolor}

\newcommand{\Eq}[1]{(\ref{eq:#1})}

\newcommand{\Lem}[1]{Lem.~\ref{lem:#1}}
\newcommand{\Con}[1]{Conj.~\ref{con:#1}}

\newcommand{\Sec}[1]{\S \ref{sec:#1}}
\newcommand{\Fig}[1]{Fig.~\ref{fig:#1}}
\newcommand{\Tbl}[1]{Table~\ref{tbl:#1}}
\newcommand{\App}[1]{Appendix~\ref{app:#1}}

\newcommand{\InsertFig}[4]
{\begin{figure}[h!t]
       \centerline{
         \includegraphics[width=#4\columnwidth]{./figures/#1}
       }
       \caption{{\footnotesize  #2}
       \label{fig:#3}}
\end{figure}}

\newcommand{\InsertFigTwo}[5] {
\begin{figure}[h!t]
       \centerline{
         \includegraphics[width=#5\columnwidth]{./figures/#1}
         \hskip 0.5in
         \includegraphics[width=#5\columnwidth]{./figures/#2}
       }
       \caption{{\footnotesize  #3}
       \label{fig:#4}}
\end{figure}}
\newcommand{\InsertFigOneThree}[8] {
\begin{figure}[h!t]
       \centerline{
\renewcommand{\arraystretch}{0.01}
         \begin{tabular}{cc}
         \includegraphics[width=#8\columnwidth]{./figures/#1} \\  \includegraphics[width=#7\columnwidth]{./figures/#2} 
         \hskip 0.15in
        \includegraphics[width=#7\columnwidth]{./figures/#3}  
        \hskip 0.15in
        \includegraphics[width=#7\columnwidth]{./figures/#4}
        \end{tabular}
       }
       \caption{{\footnotesize  #5}
       \label{fig:#6}}
\end{figure}}
\newcommand{\InsertFigFour}[7] {
\begin{figure}[h!t]
       \centerline{
\renewcommand{\arraystretch}{0.01}
         \begin{tabular}{cc}
         \includegraphics[width=#7\columnwidth]{./figures/#1} &  \includegraphics[width=#7\columnwidth]{./figures/#2} \\
        \includegraphics[width=#7\columnwidth]{./figures/#3}  &  \includegraphics[width=#7\columnwidth]{./figures/#4}
        \end{tabular}
       }
       \caption{{\footnotesize  #5}
       \label{fig:#6}}
\end{figure}}
\newcommand{\InsertFigSix}[9] {
\begin{figure}[h!t]
       \centerline{
\renewcommand{\arraystretch}{0.01}
         \begin{tabular}{cc}
         \includegraphics[width=#9\columnwidth]{./figures/#1} &  \includegraphics[width=#9\columnwidth]{./figures/#2} \\
        \includegraphics[width=#9\columnwidth]{./figures/#3}  &  \includegraphics[width=#9\columnwidth]{./figures/#4} \\
        \includegraphics[width=#9\columnwidth]{./figures/#5}  &  \includegraphics[width=#9\columnwidth]{./figures/#6}
        \end{tabular}
       }
       \caption{{\footnotesize  #7}
       \label{fig:#8}}
\end{figure}}

\newcommand{\bN}{{\mathbb{ N}}}

\newcommand{\bR}{{\mathbb{ R}}}

\newcommand{\bZ}{{\mathbb{ Z}}}
\newcommand{\cA}{{\mathcal{ A}}}
\newcommand{\cC}{{\cal C}}
\newcommand{\cB}{{\cal B}}

\newcommand{\cF}{{\cal F}}
\newcommand{\cG}{{\cal G}}
\newcommand{\cL}{{\cal L}}

\newcommand{\cR}{{\cal R}}

\newcommand{\eps}{\varepsilon}

\newcommand{\hen} {H\'enon }
\DeclareMathOperator{\sign}{sign}

\newtheorem{thm}{Theorem}
\newtheorem{lem}[thm]{Lemma}

\newtheorem{con}[thm]{Conjecture}

\usepackage{enumitem}
\usepackage{hyperref}
\newcounter{dummy}
\makeatletter 
    \newcommand\myitem[1][]{\item[#1]\refstepcounter{dummy}\def\@currentlabel{#1}}
\makeatother

\newcommand{\beq}[1]{\begin{equation}\label{eq:#1}}
\newcommand{\eeq}{\end{equation}}

\newenvironment{se}[1]{\equation\label{eq:#1}\aligned}{\endaligned\endequation}
\newcommand{\bsplit}[1]{\begin{se}{#1}}
\newcommand{\esplit}{\end{se}}

\newenvironment{example}[1][]
  {
	\setlength \leftmargini {1.0em}		%controls the indentation for quote & list
	\setlength \topsep {0.5em}			%controls the top space for quote & list
	\begin{quote}
	{\it Example#1} }
	{\end{quote}
  }
\newcommand{\bexam}[1][:]{\begin{example}[#1]}
\newcommand{\eexam}{\end{example}}

\title{Connecting Anti-integrability to Attractors for Three-Dimensional, Quadratic Diffeomorphisms}
\author{Amanda E. Hampton and James D.~Meiss\thanks
      {
        The authors were supported in part by NSF grant DMS-181248 and a donation from Northrup Grumman.
        Useful conversations with Yi-Chiuan Chen are gratefully acknowledged.
      }
    \\
 \begin{tabular}{c}
	Department of Applied Mathematics\\
    University of Colorado \\
	Boulder, CO 80309-0526 \\
	Amanda.Hampton@colorado.edu\quad
	James.Meiss@colorado.edu\\ 
\end{tabular}
}

\date{\today}

\begin{document}

\maketitle

\begin{abstract}
We previously showed that three-dimensional quadratic diffeomorphisms have anti-integrable (AI) limits that correspond to a quadratic correspondence; a pair of one-dimensional maps. At the AI limit
the dynamics is conjugate to a full shift on two symbols. Here we consider a more general AI limit, allowing two parameters of the map to go to infinity. We prove the existence of AI states for each symbol sequence for three cases of the quadratic correspondence: parabolas, ellipses and hyperbolas. A contraction argument gives parameter domains such that this is a bijection, but the correspondence also is observed to apply more generally. We show that orbits of the original map can be obtained by numerical continuation for a volume-contracting case. These results show that periodic AI states evolve into the observed periodic attractors of the diffeomorphism. We also continue a periodic AI state with a symbol sequence chosen so that it continues to an orbit resembling a chaotic attractor that is a  3D version of the classical 2D \hen attractor.
\end{abstract}

\tableofcontents %Turn off for JMP

%%%%%%%%%%
%%%%% Introduction
%%%%%%%%%%%
\section{Introduction}\label{sec:Intro}

While dynamical systems can have \textit{simple} behavior when they are integrable, they can also behave relatively simply at a complementary limit, where they are anti-integrable (AI) \cite{Aubry90}. For the former, all orbits lie on invariant tori determined by a complete set of invariants of the system, and for the latter all orbits are represented by a Bernoulli shift on a finite set of symbols---a pure form of chaos. Indeed, just as one can continue from an integrable system to one that is nearly integrable to find persistent tori using KAM theory, one can also continue from an AI limit to find chaotic dynamics, under an assumed nondegeneracy condition \cite{MacKay92b, Aubry95,Sterling98}.
These two methods, while fundamentally different, can both help elucidate
the more difficult, intermediate case where their can be a complex mixture
of both regular and chaotic motion.

In this paper we continue our application of these ideas to the study of quadratic diffeomorphisms \cite{Hampton22,Hampton22b}. Our work is based on a similar analysis for the 2D \hen map \cite{Sterling98,Sterling99,Dullin05, Chen06},
but extended to three-dimensional quadratic diffeomorphisms \cite{Lomeli98}.
These quadratic maps have a two-symbol AI limit, and upon continuation, each symbol sequence becomes an orbit of the deterministic map. The simplest continuation argument uses the contraction mapping theorem, but numerical methods can continue these orbits beyond the parameter regions that the theorem applies.

Anti-integrability in higher-dimensional systems, written as scalar difference equations, has also been
developed by \cite{Du06,Li06,Juang08,Li10} and has been used to extend \cite{MacKay92b} to more
general multi-dimensional cases \cite{Chen15}. We compared our approach to these ideas in \cite{Hampton22}.
Our goal in the current paper is to extend these arguments to a more general anti-integrable limit that includes two limiting parameters.

We study the 3D quadratic diffeomorphism $L:\bR^3 \to \bR^3$ with quadratic inverse which,
as was shown in shown in \cite{Lomeli98, Lenz99}, can be written in the form
\bsplit{Q3DMap}
	L(x,y,z) &= (\delta z+ \alpha + \tau x - \sigma y + Q(x,y), x,y), \\
	Q(x,y) &\equiv ax^2 + b xy + cy^2.
\esplit
The orbits of \Eq{Q3DMap} are sequences $\{(x_t,y_t,z_t): t \in \bZ\}$ that satisfy
\[
   (x_{t+1},y_{t+1},z_{t+1}) = L(x_{t},y_{t},z_{t}).
\]
The map has seven parameters $(\alpha,\tau,\sigma,a,b,c)$, and $\delta = \det{DL}$, the Jacobian determinant.
The volume-preserving case, $\delta = 1$, is as a normal form near a fixed point with a triple-one multiplier \cite{Dullin08a}.  The volume-contracting case, $|\delta| < 1$, arises as a normal form near homoclinic bifurcations \cite{Gonchenko06} and can give rise to discrete Lorenz-like attractors \cite{Gonchenko21a}. 

We introduce the convention described in \cite{Lomeli98}: assuming $a+b+c\neq 0$ and $2a+b\neq0$,
%\footnote
%%%%
%{If one of these is violated, other scaling transformations can be found to eliminate two of the parameters.}
%%%
then an affine coordinate transformation allows one to set 
\beq{ParamSpace}
    a+b+c = 1 \mbox{ and } \tau = 0 .
\eeq
We adopt this simplification so the map only depends on $(\alpha, \sigma, a, c)$, and the Jacobian $\delta$.\footnote
{If $c=1$ and $a=b=0$, then it would be more appropriate to assume $b+2c\neq0$ and set $\sigma=0$.}

Following \cite{Hampton22}, to set up the anti-integrable limit, first rewrite \Eq{Q3DMap} as a third-order difference equation for $\{x_t: t \in \bZ\}$ upon noting that $y_{t+1} = x_t$ and $z_{t+1} = x_{t-1}$,
\beq{VPDiff}
	x_{t+1} = \delta x_{t-2} + \alpha  - \sigma x_{t-1} + Q(x_t,x_{t-1}).
\eeq
It is convenient to then rescale the phase space variables, defining $\xi_t = \eps x_t$, thus introducing a
parameter $\eps$. Then, using \Eq{ParamSpace}, \Eq{VPDiff} becomes
\bsplit{Rescaled}
    0 & = \cL_\eps(\xi_{t+1},\xi_t,\xi_{t-1},\xi_{t-2}) \\
      &= Q(\xi_t, \xi_{t-1}) + \eps^2 \alpha - 
          \eps(\xi_{t+1} + \sigma \xi_{t-1}- \delta \xi_{t-2}).
\esplit

Generally a difference equation like \Eq{Rescaled} has an AI limit if it degenerates in some way---say to a lower-order  equation, such that it is no longer a deterministic map. Such a limit is useful if (a) the orbits at the AI limit can be simply characterized (for example by symbolic dynamics), and (b) if each such orbit can be shown to continue away from the limit (for example by a contraction mapping argument or by numerical continuation). This is the method introduced in \cite{Sterling98} for the \hen map, that we extended in \cite{Hampton22} for \Eq{Rescaled}.

To get such a limit we will assume that $a,b,c$ are ``structural'' parameters that remain finite and that $|\delta| \le 1$ so that the map is not volume expanding.
Different AI limits can be categorized by scaling the remaining parameters $\alpha$ and $\sigma$ with $\eps$ so that the third-order difference equation \Eq{Rescaled} degenerates to a lower-order, non-deterministic system at $\eps = 0$. In \cite{Hampton22}, the parameter $\eps$ 
was defined by $\alpha\eps^2 = 1$. In this case $\eps \to 0$ corresponds to the anti-integrable limit $\alpha \to -\infty$. Here we will allow two parameters to be unbounded: $\alpha \to -\infty$ and $\sigma \to \pm\infty$. Setting
\beq{rDefine}
    \alpha=-\eps^{-2}, \quad \sigma= r \eps^{-1} ,
\eeq
so that $\sigma^2/\alpha = - r^2$ is finite, results in the difference equation
\beq{RescaledDifEq}
    \cL_\eps(\xi_{t+1},\xi_{t},\xi_{t-1},\xi_{t-2}) 
      = Q(\xi_t,\xi_{t-1})  - r\xi_{t-1} - 1 - \eps(\xi_{t+1}-\delta \xi_{t-2}),
\eeq
When $\eps =  0$, \Eq{RescaledDifEq} degenerates to
\beq{AILimit}
    Q(\xi_t,\xi_{t-1})= r\xi_{t-1} + 1,
\eeq
The implication is that sequential points must lie on a quadratic curve, which is determined by the discriminant,
\beq{DeltaDefine}
    \Delta \equiv b^2-4ac.
\eeq
For $\Delta<0$, \Eq{AILimit} defines an ellipse in the $(\xi_{t-1},\xi_t)$-plane, $\Delta=0$ gives a parabola,
and $\Delta>0$, a hyperbola. Note that when $\Delta = -4ac =ar^2$, the curve \Eq{AILimit} becomes degenerate in the sense that the parabola turns into a pair of parallel lines, the hyperbola becomes two intersecting lines, and the ellipse becomes a single point. The discriminant will play a crucial role in the analysis presented in \Sec{AICase}.

For the remainder of our discussion we assume $a\neq 0$ and set $b=0$, so that $a=1-c$ and $\Delta=-4ac$ now only depends on the value of $c$. A proof is given in \Sec{Existence} for the existence of AI states, which uses a similar contraction mapping argument found in \cite{Hampton22}, but results in a larger parameter region. The remainder of \Sec{Existence} gives numerical results for the region of existence of AI states in $(r,c)$-space. In \Sec{AnalyticalBounds} we will see that this region appears to converge to a simpler one that we compute analytically. We present details of the computation of this analytical region for the three classes of quadratic curve.

In \Sec{SCcase} we focus on a ``strongly contracting'' case, where $\delta = 0.05$,
that we also studied in \cite{Hampton22b}. 
Since the Jacobian determinant is small, the orbits of this map are close to those of a \hen map that \Eq{Q3DMap} reduces to at $\delta = 0$.
In \cite{Hampton22b} we found periodic, regular aperiodic, and chaotic \hen-like attractors for this case for   regions in $(\alpha,\sigma)$. The resonances that correspond to periodic attractors mimic the Arnold-tongues seen for circle maps and the ``shrimps'' seen in two parameter families of 1D maps \cite{Gallas94,MacKay87f,Facanha13}. In \Sec{SCcase}, we redo these computations using the parameters   $(\alpha,r)$. 

We then continue the AI states away from $\eps = 0$, to see how, at least for low periods, they evolve into the observed attractors. In each case we see that the attractor can be attributed to an orbit that connects to the AI limit---supporting the ``no-bubbles'' conjecture of \cite{Sterling99}.
Lastly, the continuation method is used to find periodic approximations to a \hen-like attractor. The resulting AI state for the periodic approximation continues to an orbit that is a good approximation of the attractor.

%%%%%%%%%%
%%%%% AL Limit
%%%%%%%%%%%
\section{Two-Parameter Anti-integrable Limit}\label{sec:AICase}

When $b=0$, $a=1-c$ and \Eq{AILimit} becomes
\beq{ReducedAILimit}
    a\xi_t^2+c\xi_{t-1}^2 =r\xi_{t-1}+1,
\eeq
a quadratic curve symmetric about the horizontal axis and centered at
\beq{ConicCenter}
    \xi^o=\tfrac{r}{2c}.
\eeq
% \beq{ReducedAILimit}
%     a\xi_t^2+c(\xi_{t-1} - \tfrac{r}{2c})^2 = 1 + \tfrac{r^2}{4c},
% \eeq
The discriminant reduces to $\Delta = 4c(c-1)$, so this quadratic curve is simply determined by the value of c:
\bsplit{CurveCategories}
\text{Ellipse: } \Delta<0 &\implies 0<c<1 \\
\text{Parabola: } \Delta=0 &\implies c=0 \text{ and } c=1 \\
\text{Hyperbola: } \Delta>0 &\implies c<0 \text{ and } c>1
\esplit
In the $(r,c)$ plane these correspond to the regions sketched in \Fig{RCPlane}: blue for ellipses, tan for hyberbolas, and the bounding black lines for parabolas.  Recall that $\Delta = -4ac = ar^2$ leads to degeneracies. Here, there are three cases: when $ c= 1-a = 1$,  \Eq{ReducedAILimit} becomes a pair of vertical lines (recall that we assume $a\neq 0$, though we do consider the $a=0$ case for the backwards map); when $c= -\tfrac{r^2}{4}$, \Eq{ReducedAILimit} becomes a pair of intersecting lines; and finally when $c=r=0$, a pair of horizontal lines. In \Fig{RCPlane} these are the red, dashed curves. Finally, insets in this figure show examples of the quadratic curve \Eq{ReducedAILimit} in each $(r,c)$ region or boundary.

%%%%%%%%%
\InsertFig{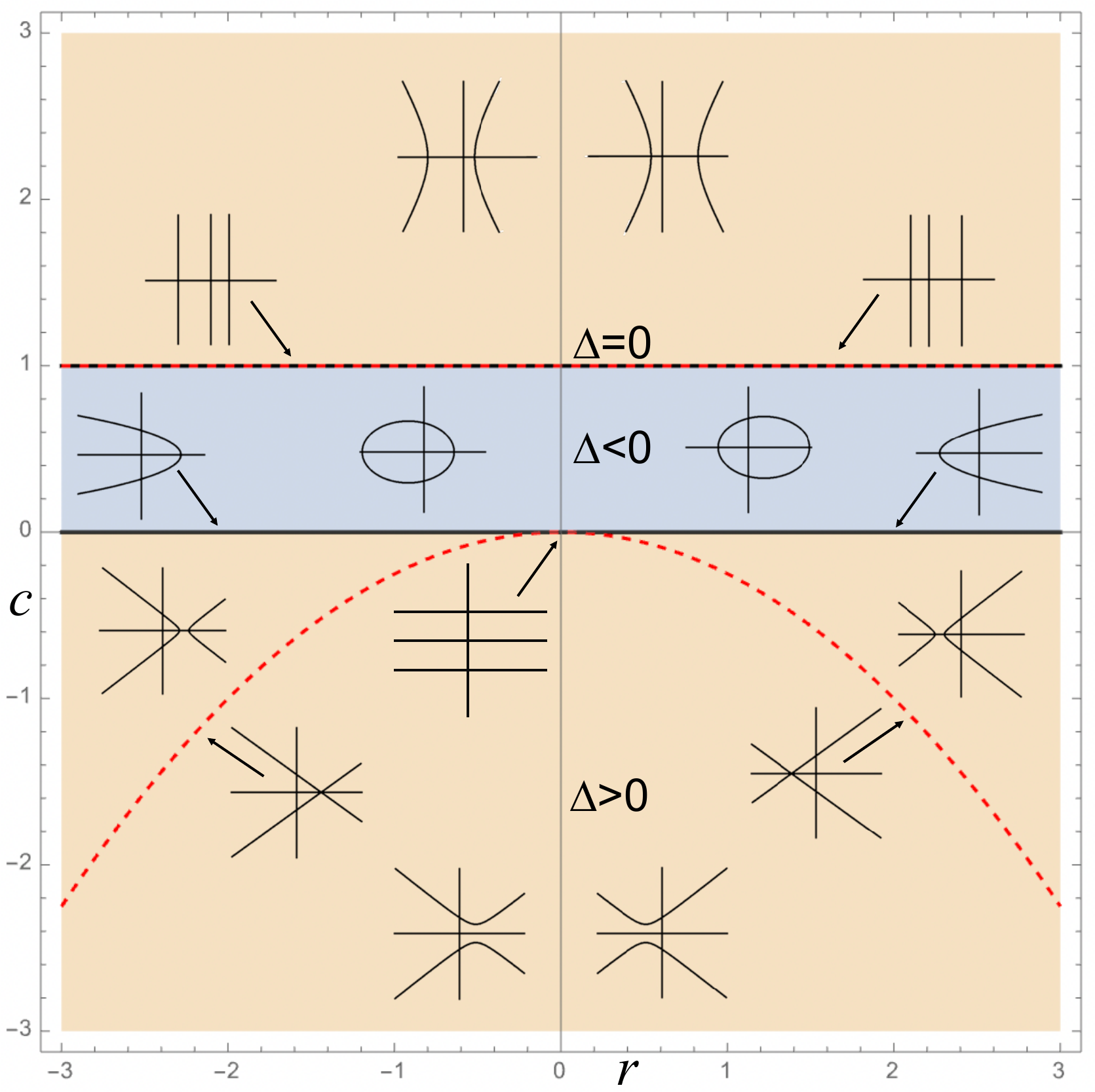}{Regions in the $(r,c)$-plane that correspond with the AI curve \Eq{AILimit} classification: elliptic (blue) with $\Delta<0$, hyperbolic (tan) with $\Delta>0$, and parabolic (black lines) with $\Delta=0$. The red-dashed curves correspond to degeneracies. Insets are included of the curve \Eq{AILimit} to show what type of curve \Eq{AILimit} corresponds to with different $(r,c)$-values.}{RCPlane}{0.6}
%%%%%%%%%

An AI state is a sequence $\{\xi_t: t\in \bZ\}$ that lies on the curve \Eq{AILimit} for all $t \in \bZ$. 
Since the map \Eq{Q3DMap} is three-dimensional, the curve \Eq{AILimit} represents a surface in $\bR^3$. This relation must hold for all $t$. Therefore, if the axes are labeled as $(\xi_{t-1},\xi_{t},\xi_{t+1})$, AI states must lie on the intersection of the two surfaces
\[
    \left\{Q(\xi_t,\xi_{t-1})=r\xi_{t-1}+1 \right\} \cap \left\{Q(\xi_{t+1},\xi_t)=r\xi_t+1 \right\}.
\]
For example, when $\Delta <0$ this is the intersection of two elliptical cylinders \cite{Hampton22}.

Dynamically, \Eq{ReducedAILimit} can be thought of as a non-deterministic quadratic correspondence, and its solutions, provided they exist, are AI states. These can be obtained by solving \Eq{ReducedAILimit} for $\xi_t$, giving a pair of 1D maps:
\beq{AIMap}
    \xi_t = f_{s_{t}}(\xi_{t-1}) = 
         s_t\sqrt{\frac{-c \xi_{t-1}^2+r\xi_{t-1}+1}{a}} , \quad s_t \in \{-,+\} ,
%       = \frac{1}{2a}\left(-b\xi_{t-1}+ 
%          s_t\sqrt{\Delta \xi_{t-1}^2+4ar\xi_{t-1}+4a}\right), \quad s_t \in \{-,+\} .
\eeq
since we have assumed $a \neq 0$.\footnote
{If $a = 0$ the set \Eq{ReducedAILimit} is a pair vertical lines at the fixed points \Eq{FixedPt} and an
AI state is simply a sequence of these points $\xi_t = \xi_\pm$.}

If the radicand of \Eq{AIMap} is strictly positive, either choice $s_t = \pm$ is valid, and each point $\xi_{t-1}$ has
two images defined by the branches, $f_-$ and $f_+$, respectively.  
Note that $f_+ \ge 0$ and $f_- \le 0$ and the maps are symmetric, $f_-(\xi_{t-1})= -f_+(\xi_{t-1})$. 
Consequently, whenever $\xi_t \neq 0$ is an AI state, it has a unique symbol sequence
\beq{SigmaDefine}
    s = \{\ldots s_0,s_1,s_2\ldots \}  \in \Sigma \equiv \{- ,+\}^\infty ,
\eeq
so that $s_t = \sign(\xi_t)$ represents the branch of \Eq{AILimit} at time $t$. 
For example, the fixed points of \Eq{AIMap} correspond to the symbol sequences $s=\{+\}^\infty$ 
and $\{-\}^\infty$, and are
\beq{FixedPt}
    \{\pm\}^\infty: \xi_{\pm} =\tfrac12(r\pm\sqrt{r^2+4}).
\eeq
%Note that the argument of the radical is always positive, so the fixed points always exist. The period-2 orbit solves $Q(\xi_1,\xi_0)=r\xi_0+1$ and $Q(\xi_0,\xi_1)=r\xi_1+1$, and corresponds with the symbol sequence $\{-,+\}^\infty$,
% \beq{Pd2}
% \{-,+\}^\infty: \xi_{0,1}=\frac{1}{2(2c-1)}\left(r \pm \sqrt{4+16c^2-3r^2+4c(r^2-4)}\right).
% \eeq
When $r=0$, a case we previously studied \cite{Hampton22b}, the fixed points are simply the symbol sequence themselves, i.e., $\xi_t=s_t$.
Fixed points exist for any $(r,c)$, and are the unique orbits with the symbol sequences $\{+\}^\infty$ and $\{-\}^\infty$.

 %Of course at this stage, we cannot say that an AI state exists for each sequence $s$, nor can we assert that if it exists, it is unique; these questions will be addressed in \Sec{Existence}.

%%%%%%%%
%%% Existence
%%%%%%%%
\section{Existence of AI states}\label{sec:Existence}

In this section, we will extend the results of \cite{Hampton22} to the case $r \neq 0$ for \Eq{ReducedAILimit} on the existence and uniqueness of AI states.
In order that a point has a forward orbit under the relation \Eq{AIMap},
it is sufficient for the range to be be a subset of the domain.
In addition, so that each AI state have an unambiguous symbol sequence, we require that the image not include the origin.
This is equivalent to requiring the radicand of \Eq{AIMap} to be strictly positive.
Thus we suppose there exists a set $B \subset \bR$ so that
\beq{MapsInto}
    f_\pm(B) \subset B \setminus \{ 0 \}.
\eeq

The correspondence between symbol sequences and AI states becomes bijective when the derivative,
\beq{AIMapDeriv}
f'_{s_t}(\xi_{t-1})
            %=\frac{s_t}{2a}(r-2c\xi_{t-1})\sqrt{\frac{a}{-c\xi_{t-1}^2+r\xi_{t-1}+1}},
             = \frac{s_t(r-2c\xi_{t-1})}{2\sqrt{a(-c\xi_{t-1}^2+r\xi_{t-1}+1)}}.
\eeq
has magnitude less than one on $B$. Indeed, when $B$ is compact, as we noted in \cite{Hampton22},
the contraction mapping theorem implies this bijection: 
for each $s \in \Sigma$ there is a unique $\{\xi_t\} \in B^\infty$ satisfying \Eq{AIMap}.

In \Sec{OneToOne} we generalize this result, requiring only that the $n$-step composition has an absolute slope 
less than one for some $n \in \bN$. 
This allows for points of the orbit to lie on portions of the curve \Eq{ReducedAILimit} with a `steeper' slope, 
thus increasing the parameter regime where we can prove the existence of AI states. In \Sec{NumericalBounds}
we compute the $n$-step bounds numerically.

%%%%%%%%
%%% Sufficient
%%%%%%%%
\subsection{One-to-One Correspondence}\label{sec:OneToOne}

Suppose that $B \subset \bR$ is compact, and that $\xi = \{\ldots \xi_{-1}, \xi_0, \xi_1,\ldots\}$ denotes a sequence 
in
\beq{BDefine}
    \cB=B^\infty\subset \bR^\infty,
\eeq
the countable product of $B$. Define $\cF: \cB \times \Sigma \to \cB$ by
\beq{AIMapF}
    \cF_t(\xi; s) = f_{s_t}(\xi_{t-1}),
\eeq
for each $s\in\Sigma$ \Eq{SigmaDefine} and $t \in \bZ$.
Note that any fixed point of $\cF$ is an orbit of \Eq{AIMap} with sequence $s$.
More generally, for each $k \in \bN$, let $\cF^k: \cB \times \Sigma \to \cB$ denote the $k^{th}$ iterate,
\[
    \cF^k_t(\xi, s) = f_{s_{t}} \circ f_{s_{t-1}} \circ \ldots \circ f_{s_{t-k+1}}(\xi_{t-k}) 
\]
For example, if $s$ is a period-$k$ symbol sequence (e.g., $s_{t+k} = s_t, \forall t \in \bZ$) and $\xi$ is a period-$k$ orbit of \Eq{AIMap}, then it is a  fixed point of $\cF^k$. 

%\[
%D\cF (\xi_{0}) =\lim_{n\to\infty} \prod_{j=0}^{n} f'_{s_{j+1}}(\xi_{j})=
%\frac{d}{d\xi_{t-1}}[f^n_s(\xi_{t-1})]= \frac{d}{d\xi_{t-1}}[f_{s_{t+n}}(f_{s_{t+n-1}}( \ldots f_{s_t}(\xi_{t-1}) \ldots ))] = 
%\lim_{n\to\infty} \prod^{n}_{j=0}\frac{s_{j+1}(r-2c\xi_{j})}{2\sqrt{a(-c\xi_{j}^2+r\xi_{j}+1)}}.
%\]
Lastly, for $n\in\bN$, define the set of parameters
\beq{RnPlus}
    \cR_n^+ = \{(r,c): \exists\, B \neq \emptyset, \, f_{s_t}(B) \subset B, \, 
       \|D\cF^n\|_\infty<1 \},
\eeq
where
\beq{DFn}
    \|D\cF^n\|_\infty = \max_{s_0,s_1,\ldots s_{n-1}} \sup_{x\in B} \left| \tfrac{d}{dx}(f_{s_{n-1}}(f_{s_{n-2}}(\ldots (f_{s_0}(x))\ldots))) \right|,
\eeq
i.e., the $\infty$-norm  on $\cB \times \Sigma$. 
A simple extension of the result in \cite{Hampton22} is the following.

%%%%%%%%%
\begin{lem}\label{lem:AIGeneralContraction} 
Given $(r,c) \in \cR_n^+$ and $\xi\in\cB$, there is a one-to-one correspondence between each sequence $s \in \Sigma$ 
and state $\{\xi\} \subset \cB$ satisfying \Eq{AIMap}.
\end{lem}
%%%%%%%%

\begin{proof}
Consider a pair of sequences $\xi, \eta \in \cB$. Given some $s\in\Sigma$ and $\cF$ \Eq{AIMapF}, the fixed points of $\cF$ are orbits of the map \Eq{AIMap}. Since $B$ is a compact subset of $\bR$, and $\cB$ is complete in the $\ell^\infty$ norm. Then
\[
    \| \cF^n(\xi) -\cF^n(\eta)\|_\infty \le ||D\cF^n||_\infty \, \|\xi - \eta\|_\infty,
\]
so that for parameters in $\cR_n^+$,  $\cF^n$ is a contraction. Thus $\cF^n$ has a unique fixed point $\xi^* = \cF^n(\xi^*)$.  

When $(r,c) \in \cR_n^+$, $f_{s_t}(B) \subset B$; therefore, for each $s \in \Sigma$, there is an $\eta^* \in \cB$ such that
for each $t\in \bZ$, $\cF_t(\eta^*,s)= \eta^*_t$, since this is simply an orbit of the 1D map \Eq{AIMap}. 
Since $\eta^*=\cF(\eta^*)$, clearly we also have $\eta^*=\cF^n(\eta^*)$. 
Since the fixed point of $\cF^n$ is unique, $\eta^*=\xi^*$. 
Hence, each symbol sequence has a unique corresponding orbit of \Eq{AIMap}. 

Note that, by construction, the conditions $(r,c) \in \cR_n^+$ and $\xi \in B$ guarantee that the radicand of \Eq{AIMap} is strictly positive, implying that every state of \Eq{AIMap} has a unique symbol sequence $s\in\Sigma$. Therefore, there exists a one-to-one correspondence between symbol sequences and states of \Eq{AIMap}.
\end{proof}
%%%%%%%%%%

Additionally, when $c \neq 0$, the argument of \Lem{AIGeneralContraction} can be applied using the backwards map,
\beq{BackwardsAIMap}
\xi_{t-1}=g_{s_t}(\xi_t)=\frac{1}{2c}\left(r + s_t \sqrt{r^2+4c-4ac\xi_t^2}\right),
\eeq
which is obtained by solving \Eq{ReducedAILimit} for $\xi_{t-1}$. Given a similar region, $\cR^-_n$, to \Eq{RnPlus}, defined using \Eq{BackwardsAIMap}, the proof of a one-to-one correspondence is the same as that for \Lem{AIGeneralContraction}.

%%%%%%%%%%
%%%%% Numerical
%%%%%%%%%%%
\subsection{Numerical Verification of AI States}\label{sec:NumericalBounds}
Of course, it can be a nontrivial task to compute the regions $\cR_n^{\pm}$ required for \Lem{AIGeneralContraction}.
Here we will assume that the set $B$ which satisfies \Eq{RnPlus} is a closed interval:
\beq{BInterval}
    B=[-\beta,\beta].
\eeq
Though this is probably not necessary, it makes the calculations simpler.
The specific choices for $\beta$ depend on $\Delta$ and will be given below in \Sec{AnalyticalBounds}.
Anticipating these results, we will compute in this section the region $\cR_n^{\pm}$ in the $(r,c)$ plane where \Eq{DFn} is achieved for $n$ up to $15$.

Figure \ref{fig:RC_PlaneAIExist}(a) shows the results of computations for a grid of $500^2$ parameter 
points in the region $|r| \le 3$ and $|c| \le 3$.
For each $(r,c)$ point, $100$ initial conditions along $B$ are iterated $n$ times using 
$f_{s_t}$ \Eq{AIMap} for $\cR_n^+$ or $g_{s_t}$ \Eq{BackwardsAIMap} for $\cR_n^-$. This is done for each of the 
$2^n$ possible symbol sequences $\{s_1,s_2, \ldots s_n\} \in \{+,-\}^n$ for $n\leq15$. 
For each iteration the derivative \Eq{DFn} is estimated as the maximum over the grid of orbits and the set of symbols.
If $\|D\cF^n\|_\infty<1$, then the parameters are added to the region $\cR_n^\pm$.
The results, shown in \Fig{RC_PlaneAIExist}(a), have a maximum of $n=12$ forward steps, and $n=14$ backward steps (i.e., using $g_\pm$).
The enlargement, in \Fig{RC_PlaneAIExist}(b), shows similar results on a $500^2$ parameter grid for
$|r| \le 1.17$ and $-0.35 \le c \le 0.75$, with a maximum of $n=15$ steps in each direction. 
In each case the color scale gives the number of forward iterates for $\cR_n^+$ (yellow to green) and backward iterates for $\cR_n^-$ (cyan to magenta). The white region in the figures corresponds to parameters for which the bound on the derivative is never attained. Of course, the regions are necessarily nested: $\cR_{n-1}^\pm \subset \cR_{n}^\pm$. Moreover, as $n$ grows the enlargement of the regions becomes relatively small and they seem to converge. We will say more about this in the next section.

Note that most of the $(r,c)$ values in the regions shown in \Fig{RC_PlaneAIExist}(a)
correspond to the hyperbolic case, either $c > 1$ for $\cR_n^-$ where the
hyperbola is more \textit{vertical} or $c< -r^2/4$ for $\cR_n^+$ where the hyperbola is more \textit{horizontal}. That the slopes should be
small for the backwards or forward map, respectively, in these cases can also be seen in the insets in \Fig{RCPlane}.
It is more delicate to satisfy the conditions for \Lem{AIGeneralContraction} for the elliptic case, $0 < c < 1$;
indeed, for an ellipse the slope is necessarily zero or infinite at the vertices, and so $B$ must be more carefully selected to avoid these points.

%%%%%%%%%%
\InsertFigFour{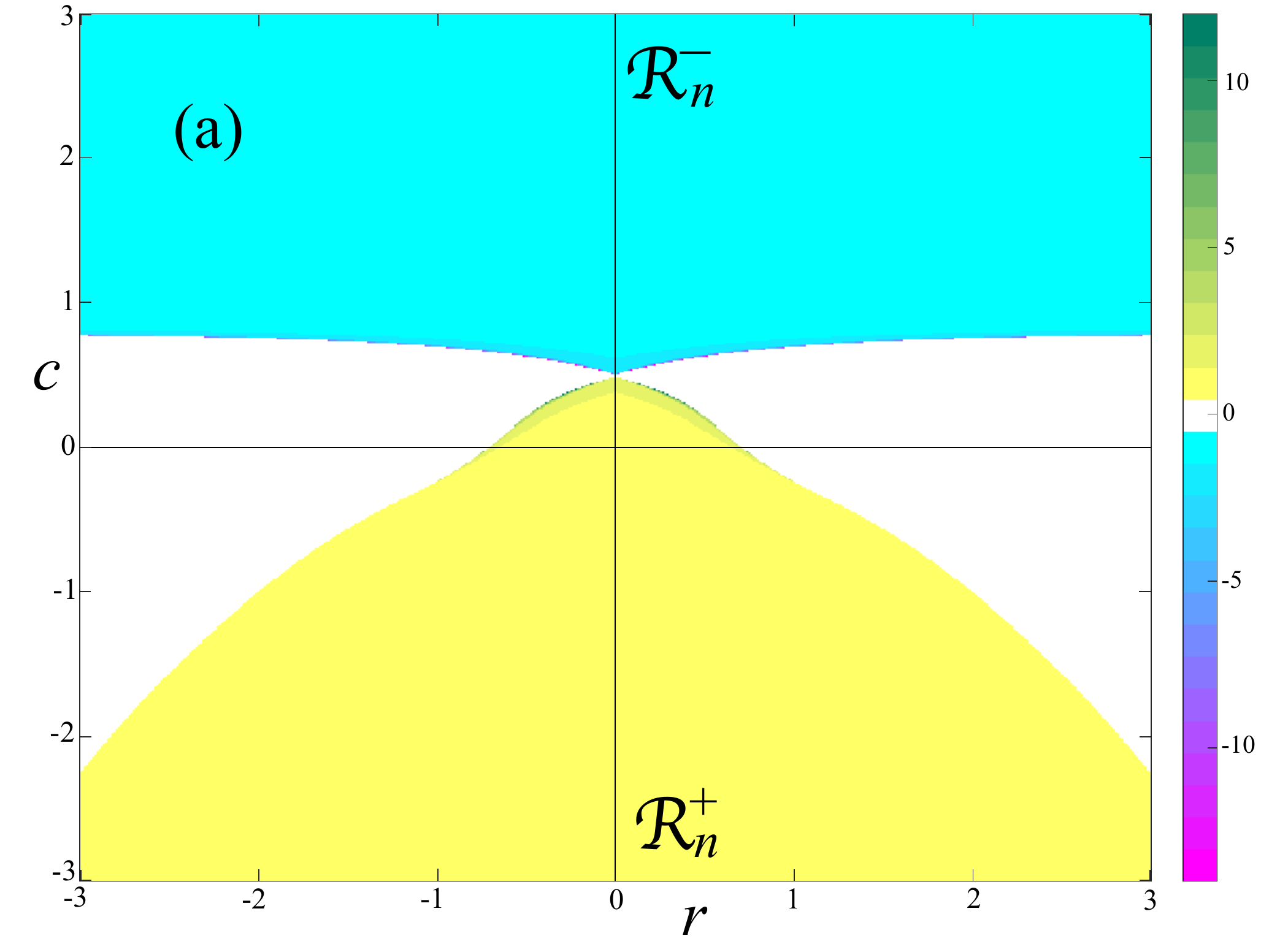}{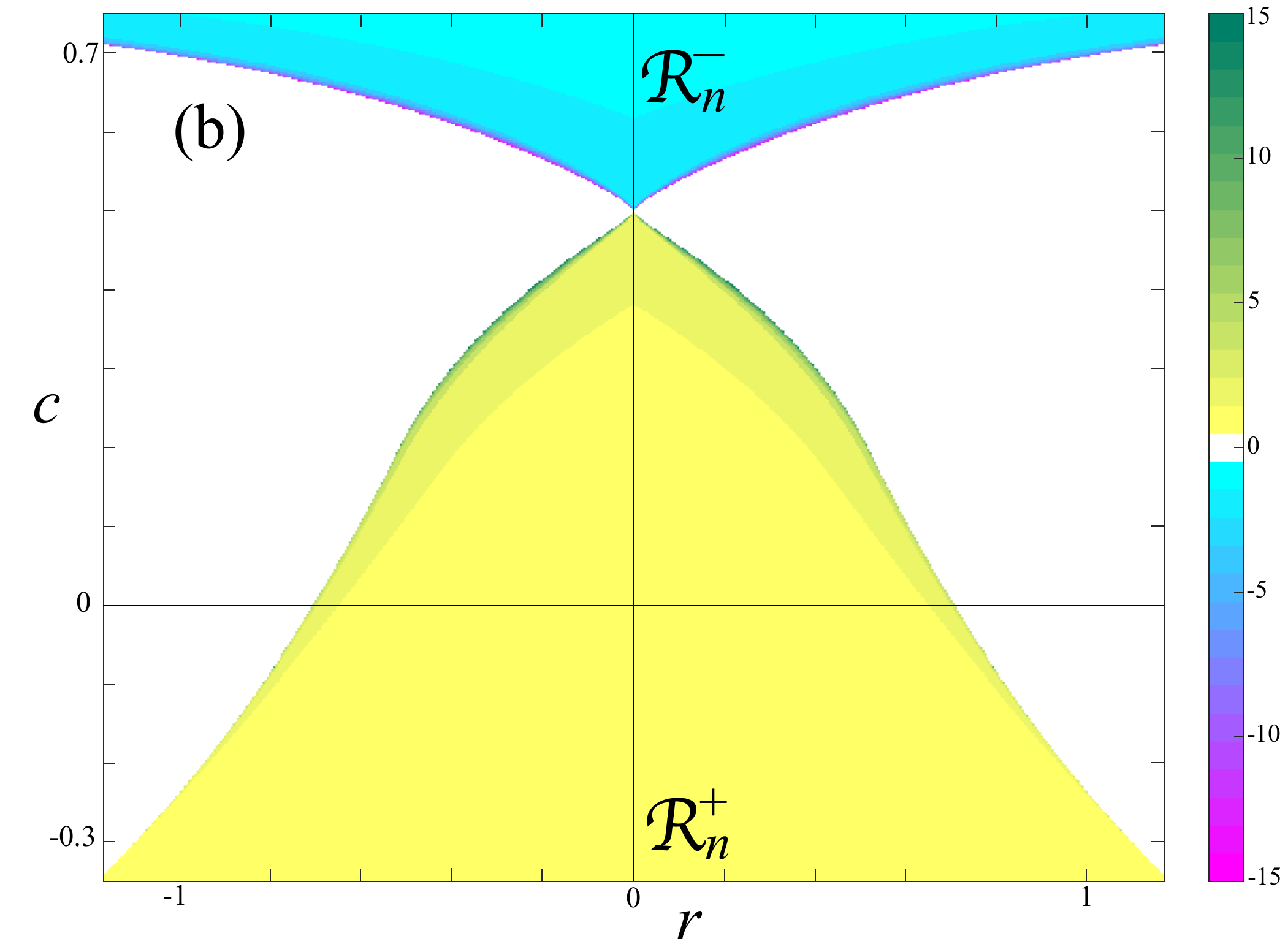}{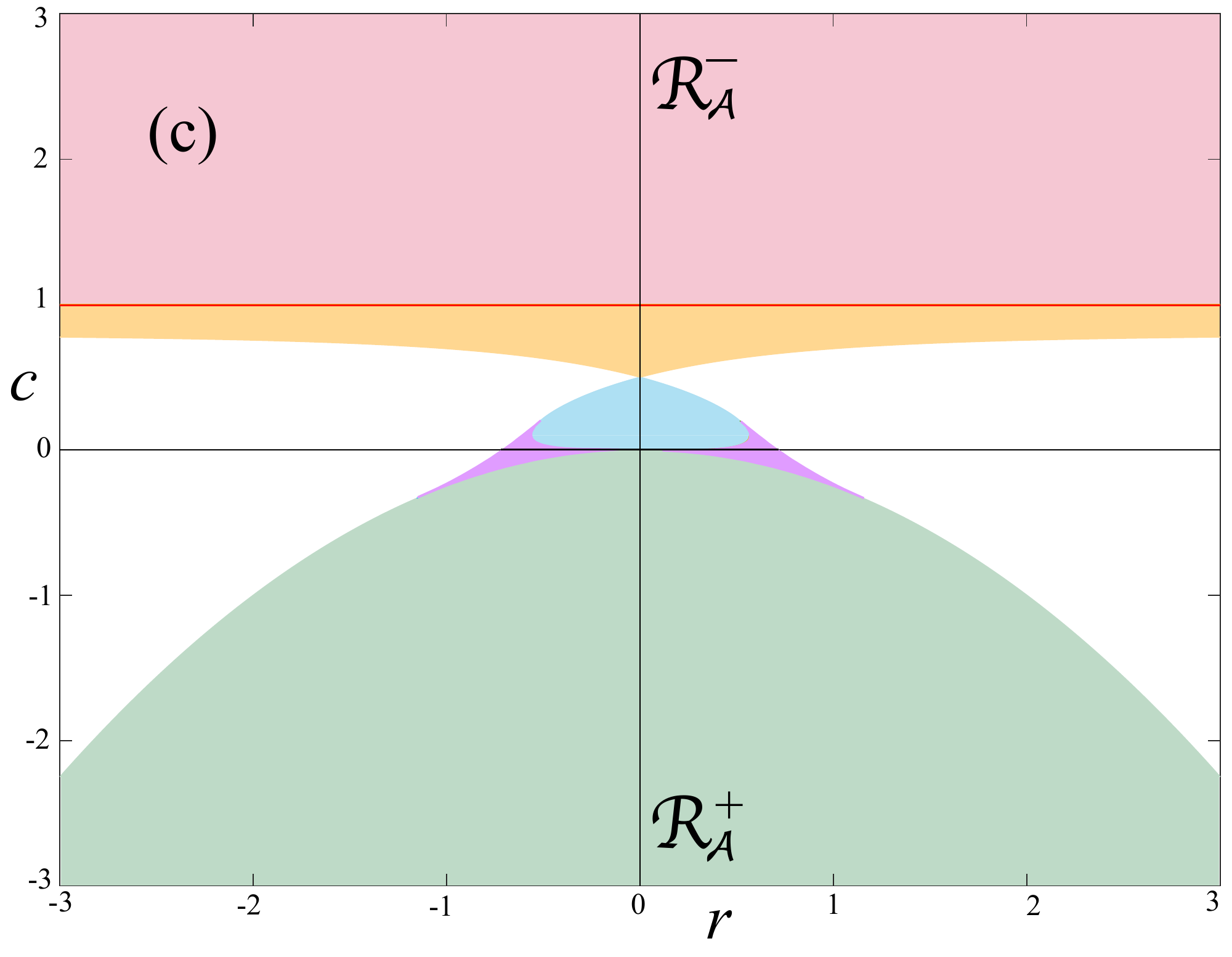}{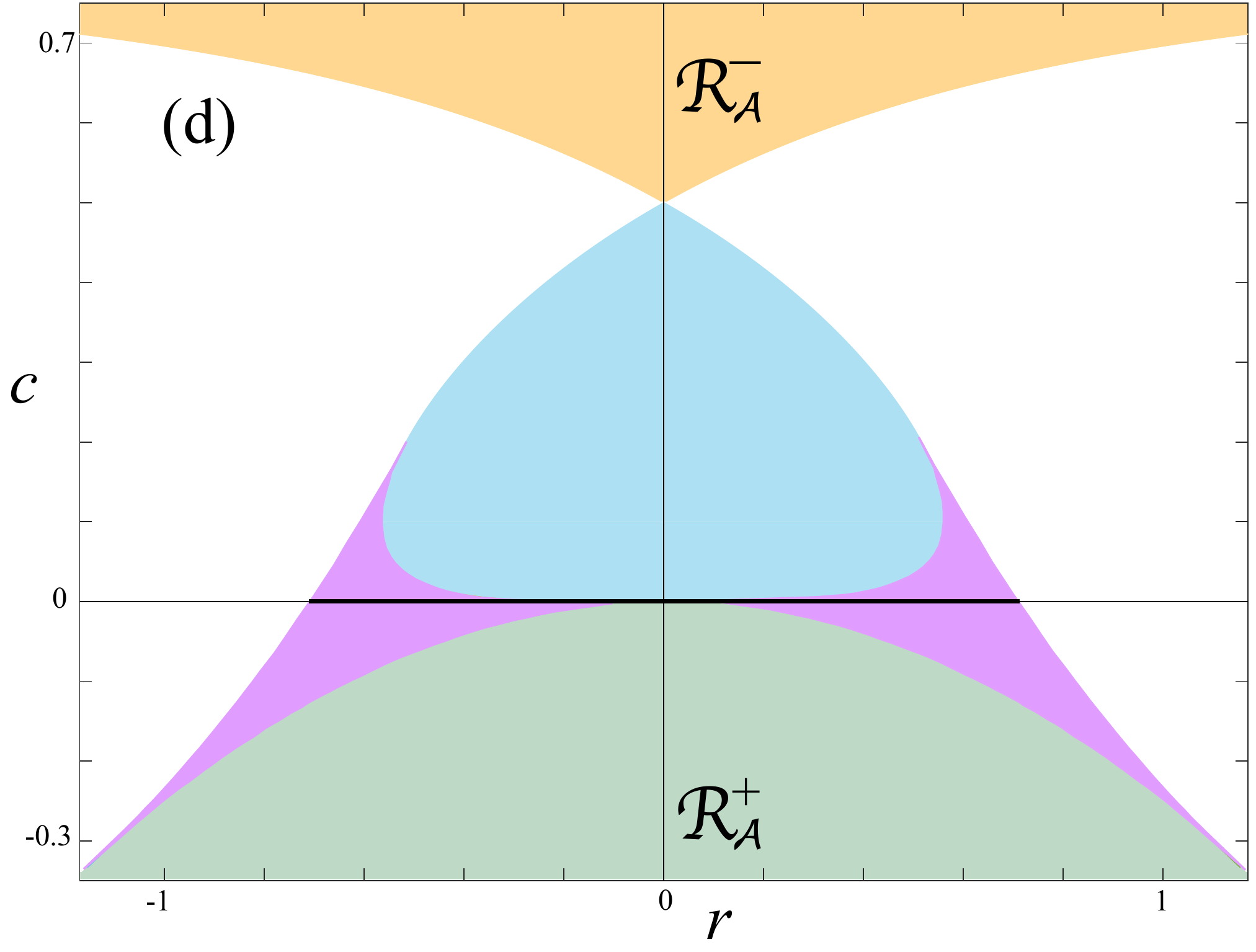}
{(a) Numerically computed approximations to $\cR_n^\pm$.
Colors correspond with the number of forward (yellow to green) or backward (cyan to magenta) iterates that are used. The maximum number of steps is $n=12$ for the forward direction and $14$ for the reverse direction. 
(b) An enlargement of (a) about the origin where $|r| \le 1.17$ and $-0.35 \le c \le 0.75$, for $n$ up to $15$ in either direction. (c) The analytical regions $\cR_\cA^\pm$ \Eq{RA} in the $(r,c)$ plane so that $f_s(B) \subset B$. for \Eq{AIMap}. Colors correspond with the different $\Delta$-cases described in the subsections of \Sec{NumericalBounds}: hyperbolic (red for $c>1$, green, and purple for $c<0$), 
elliptic (tan, blue, and purple for $0 < c < 1$) and parabolic (an interval along $c = 0$ and $c=1$).
(d) An enlargement of (c) using the same bounds as (b).}
{RC_PlaneAIExist}{0.5}
%%%%%%%%%%%

%%%%%%%%%%
%%%%% Analytical
%%%%%%%%%%%
\section{Analytical Bounds} \label{sec:AnalyticalBounds}

In this section, we report results for the parameters $(r,c)$ that \textit{only} satisfy the condition \Eq{MapsInto},
the simplest requirement so that each point has a forward or backward orbit as well as an unambiguous symbol sequence.
We denote these sets by
\bsplit{RADefine}
    \cR_\cA^+ &\equiv \{ (r,c) : f_s(B ) \subset B\setminus \{0\}, s \in \{+,-\} \} ,\\
    \cR_\cA^- &\equiv \{ (r,c) : g_s(B ) \subset B\setminus \{0\}, s \in \{+,-\} \} ,\\
 \esplit
using the forward map \Eq{AIMap} and the backward map \Eq{BackwardsAIMap}, respectively.
We will again assume that $B$ is a closed interval of the form \Eq{BInterval}.

The results will be obtained in the following subsections for the three classes of
quadratic curves: parabolas, ellipses, and hyperbolas.
These calculations lead to the forms
\beq{RA}
    \begin{array}{cllcl}
      \cR_\cA^+  &= &\left\{(r,c) : |r|\leq \tfrac{2}{\sqrt{15}} \,,\, c < \cC_2(r) \right\} 
        &\bigcup&\,   \left\{ (r,c) : \tfrac{2}{\sqrt{15}} \leq |r|\leq \tfrac{2}{\sqrt{3}} \,,\, c < 1+|r|(|r|-\sqrt{r^2+4}) \right\} \\ 
         & & & \bigcup & \,  \left\{(r,c) :   |r|\geq \tfrac{2}{\sqrt{3}} \,,\,  c< -\tfrac{r^2}{4} \right\}, \\
    \cR_\cA^- &=& \left\{(r,c) : c > \cC_3(r) \right\}
    \end{array} .
\eeq
Here, the functions $\cC_i(r)$ are roots of the cubic polynomial
\beq{EllipsePoly}
   P(c) =  64c^3+32(r^2-2)c^2+(r^2-4)(5r^2-4)c-4r^4 ;
\eeq
these are real when $|r| < 2\sqrt{2}/5$.
The calculations for $\cR_\cA^-$ are summarized \App{BWMapCalcs}. 

The regions \Eq{RA} are shown in \Fig{RC_PlaneAIExist}(c) and (d). They include subsets of the hyperbolic case (red for $c>1$, green, and purple for $c<0$), 
the elliptic case (tan, blue, and purple for $0 < c < 1$) and the parabolic case (an interval along $c = 0$ and $c=1$).

Comparing the upper and lower panels of \Fig{RC_PlaneAIExist}, it appears that the numerically found regions $\cR_n^+ $ and $\cR_n^-$ converge onto the analytically determined regions \Eq{RA} as $n$ grows.
To verify this we computed Hausdorff distances between $\cR_n^\pm$ and $\cR_\cA^\pm$, listed in \Tbl{Hausdorff} as a function of $n$.
Here we compare the sets on the domain $(r,c)=[0,\tfrac{2}{\sqrt{3}}] \times [-\tfrac13,\tfrac45]$, using a  $500^2$ grid. This region is chosen to take advantage of symmetry $r \to -r$ .
%and of the hyperbolic $\Delta>0$ case, where we know the bounds are equal for $n=1$, discussed in \Sec{Hyperbola} below. 

%%%%%%%%%%%%%%
\begin{table}[h!t]
\centering
\begin{tabular}{c|c|c}

     n  & $\| \cR_n^+ - \cR_\cA^+\|_H$  & $\| \cR_n^- - \cR_\cA^-\|_H$ \\
      \hline

      1 & 0.112941 & 0.113266  \\

      2 & 0.020430 & 0.015774   \\

      5 & 0.014207 & 0.012091  \\

      10 & 0.008460 & 0.007873  \\

      15 & 0.007008 & 0.006532  \\

\end{tabular}
\caption{Hausdorff distances between the sets $\cR_n^\pm$ and $\cR_\cA^\pm$ for increasing $n$. For larger $n$, the distances decrease, as is confirmed visually in \Fig{RC_PlaneAIExist}.}
\label{tbl:Hausdorff}
\end{table}
%%%%%%%%%%%%%%

Since the distances in \Tbl{Hausdorff} appear to go to zero, it seems reasonable to infer the following:
\begin{con}\label{con:RegionConvergence}
As $n \to \infty$, the $\cR_n^\pm$ converges to $\cR_\cA^\pm$.
\end{con}

Note that \Eq{MapsInto} requires the strict inequalities $f_+>0$ and $f_-<0$ so that each orbit has a uniquely defined symbol sequence, i.e., that $f_+(B) \cap f_-(B) = \emptyset$.
Since the range must be a strict subset of the domain, successive iterates give nested sets: $B \supset f_\pm(B) \supset f_\pm(f_\pm(B)) \ldots$. The AI states lie within the resulting Cantor-like sets. 

%%%%%%%%%%
%%%%% Parabolic
%%%%%%%%%%%
\subsection{Parabolic Case}\label{sec:Parabolic}
In this section we compute the regions \Eq{RADefine} for the parabolic case, $\Delta = 0$.
Since we assume $b=0$, this requires that $(a,c) = (0,1)$ or $(1,0)$. 
When $a=0$, the curve \Eq{ReducedAILimit} degenerates to a pair of vertical lines at the fixed points, $\xi_\pm$, \Eq{FixedPt}. 
For this case, the forward map is not defined, and one must use the backwards map, \Eq{BackwardsAIMap}, see \App{BWMapCalcs}.

%but since $g_\pm(\xi) = \xi_\pm$, we can set $B = \{\xi_-,\xi_+\}$: the backwards AI states are simply the fixed points.
% Thus
% \[
%     \{(r,1): r \in \bR\} \subset \cR_\cA^{-},
% \]
% the red line in \Fig{RC_PlaneAIExist}(c). 

For the case $c=0$, the parabola \Eq{ReducedAILimit} has a vertex at $(-\tfrac{1}{r},0)$ when $r \neq 0$ and opens in the positive (negative) $\xi_{t-1}$ direction when $r>0$ ($r<0$), recall the sketches in \Fig{RCPlane}.
%i.e.,
%\[
    %    \xi_{t-1} \leq -\frac{1}{r} \text{ for }r<0 \quad \text{and} \quad \xi_{t-1} \geq -\frac{1}{r} \text{ for %}r>0.
%\]
When $r=0$, the parabola degenerates to a pair of horizontal lines and 
$f_\pm(\xi) = \xi_\pm$, \Eq{FixedPt}. 

Since the interval $B$ must contain the fixed points, as these are orbits with $s = \{+\}^\infty$ and $\{-\}^\infty$,
and it must contain the images $f_\mp(\xi_\pm) = -\xi_\pm$, we define $B$ using
the maximum absolute fixed point $\beta = \xi_{max}$, where
\beq{MaxFixPt}
    \xi_{max}=\max(|\xi_+|,|\xi_-|) = \tfrac{1}{2}(|r|+\sqrt{r^2+4}).
\eeq
Since each branch is monotone on $B$, it is clear that $f_\pm(B) \subset B$.
To guarantee that the radicand of \Eq{AIMap} is positive, the vertex
of the parabola must be outside $B$; i.e. $-1/r < -\beta$,
this requires $|r|<\tfrac{1}{\sqrt{2}}$.
%&\frac{1}{|r|}>\beta=\frac{1}{2}(|r|+\sqrt{r^2+4})\\
%\implies & 2-r^2 >|r|\sqrt{r^2+4}>0\\
%\implies & 4-4r^2+r^4 > r^2(r^2+4)\\
% Given this bound, the maximum fixed point $\beta$ can also be bounded,
% \[
    % 1 < \beta=\frac{1}{2}(|r|+\sqrt{r^2+4})<\sqrt{2}.
% \]
This gives
\[
    \{(r,0): |r| <\tfrac{1}{\sqrt{2}}\} \subset \cR_\cA^+ .
\]
This is depicted in \Fig{RC_PlaneAIExist}(d) as the thicker, black line segment.

It's also easy to find the region $\cR^+_1$, where  $|f'_\pm(\xi)|<1$ 
when $\xi \in B$. Using \Eq{AIMapDeriv}, $|f'_\pm(\xi)| = 1$ at 
$
    %\frac{r}{2\sqrt{r\xi_{t-1}+1}}= 1 \quad \implies 
    %\implies r^2 &= 4(r\xi_{t-1}+1),\\
    \xi =\tfrac{r}{4}-\tfrac{1}{r},
$
this requires that this value is outside $B$, or equivalently
\begin{align*}
%    \left|\tfrac{r}{4}-\tfrac{1}{r}\right|>\beta \quad
    %|\frac{r}{4}-\frac{1}{r}|=|\frac{r^2-4}{4r}|&>\frac{1}{2}(|r|+\sqrt{r^2+4})\\
    %\implies |r^2-4|&>2|r|(|r|+\sqrt{r^2+4}) \quad \text{note: $r^2-4<0$ given $|r|<1/\sqrt{2}$} \\
    %\implies 4-3r^2 &>2|r|\sqrt{r^2+4}\\
     16-40r^2+5r^4>0
    %\implies 0<r^2< \frac{40 \pm 16\sqrt{5}}{10} 
    \implies \quad |r|< \sqrt{\tfrac15(20-8\sqrt{5})}\approx 0.6498394.
\end{align*}
Note that this segment is only $10\%$ smaller than that for $\cR_\cA^+$.
We extend this computation to compute $\cR_2^+$ in \App{ParabolicTwo}; this gives $|r| < 0.6984177$, now only about $1\%$ smaller than that for $\cR_\cA^+$.
These values agree with the numerical results of \Fig{RC_PlaneAIExist}(b) along $c=0$. For the calculations used to obtain the values in \Tbl{Hausdorff} (i.e. numerical results for $\cR_{15}^+$), the interval is $|r| \leq 0.7057789$, a mere $0.2\%$ smaller than $\cR_\cA^+$. For the figure the grid size is $2.314(10)^{-3}$ so that the next grid point is $r=0.70809$, which is outside $\cR_\cA^+$; thus the computed interval is the optimal.

%%%%%%%%%%
%%%%% Elliptic
%%%%%%%%%%%
\subsection{Elliptical Case}

When $0 < c < 1$, the curve \Eq{ReducedAILimit} is an ellipse centerred at \Eq{ConicCenter} and contained in the rectangle
\beq{EllipseBounds}
    \left[\tfrac{1}{2c}(r-\sqrt{r^2+4c}), \tfrac{1}{2c}(r+\sqrt{r^2+4c})\right] \times 
    \left[-\sqrt{\tfrac{r^2+4c}{4ac}},\sqrt{\tfrac{r^2+4c}{4ac}}\right] .
\eeq
There are two cases to consider. First, if the interval $[-\xi_{max},\xi_{max}]$ \Eq{MaxFixPt} 
contains the center of the ellipse \Eq{ConicCenter}, i.e., if
\beq{NotParabolicLike}
    |\xi^o| < \xi_{max},
\eeq
then, since $B$ must include the fixed points, \Eq{MapsInto} requires that
the vertical range of the ellipse must be a subset of $B$, see the sketch in \Fig{EllipticalExamples}(a).
Indeed for this case we can take
\Eq{BInterval} with $\beta$ equal to top of the rectangle \Eq{EllipseBounds}.
Requiring the range to be a strict subset of the domain gives the condition
\[
    \beta < %\min\left(\frac{1}{2c}|r\pm\sqrt{r^2+4c}|\right)=
        \frac{1}{2c}(-|r|+\sqrt{r^2+4c}).
\]
After some algebra, this implies that
\beq{EllipseRegion}
  \left\{ (r,c): |r| < \tfrac{2\sqrt{2}}{5} \,,\, \cC_1(r)< c < \cC_2(r) \right\} \subset \cR_A^+,
\eeq
where $\cC_{1,2}(r)$ are the smaller two roots of the cubic polynomial \Eq{EllipsePoly}.
The discriminant of this polynomial, $2^8r^2(r^2+4)^4(8-25r^2)$, is positive when $|r|<2\sqrt{2}/5$, 
and in this case it has three positive, real roots. This is the blue region in \Fig{RC_PlaneAIExist}(c) and (d).

When \Eq{NotParabolicLike} is not satisfied, the region defined by the fixed points does not contain the top or bottom vertices of the ellipse;
thus the slope of $f_+$ in the set $[-\xi_{max},\xi_{max}]$, \Eq{MaxFixPt}, is always positive.
An example is shown in \Fig{EllipticalExamples}(b).
For this case, as for the parabola, we can use $\beta = \xi_{max}$ for the interval \Eq{BInterval}. 
In order that $ \{0\} \notin f_\pm(B)$,
the horizontal vertices of the ellipse must be outside $B$; this gives the condition
\beq{OutlyingCondition}
    \frac{-|r|+\sqrt{r^2+4c}}{2c} > \frac{1}{2}(|r|+\sqrt{r^2+4}).
\eeq
Some more algebra then gives the additional region
%\[
%   \left\{(r,c): 0 < |r| < \tfrac{2}{\sqrt{15}} \,,\, 
%          0<c<\tfrac{1}{4} (r^2+\sqrt{r^4+4r^2})\right\} 
%    0 <c <\tfrac{1}{4} |r|(|r|+\sqrt{r^2+4})\right\} \subset \cR_A^+,
%\]
\[
   \left\{(r,c): \tfrac{2}{\sqrt{15}}<|r|<\tfrac{1}{\sqrt{2}} \,,\, 
         0 < c <1+|r|(|r|-\sqrt{r^2+4}) \right\} \subset \cR_A^+ . 
%       0<c<1+r^2-\sqrt{r^4+4r^2} \right\}. 
\]
In \Fig{RC_PlaneAIExist} (c) and (d), 
this additional portion in $0<c<1$ is colored purple. The backwards case, $\cR_A^-$, is treated in \App{BWMapCalcs}, and results in the tan region in \Fig{RC_PlaneAIExist} (c) and (d).

%%%%%%%
\InsertFig{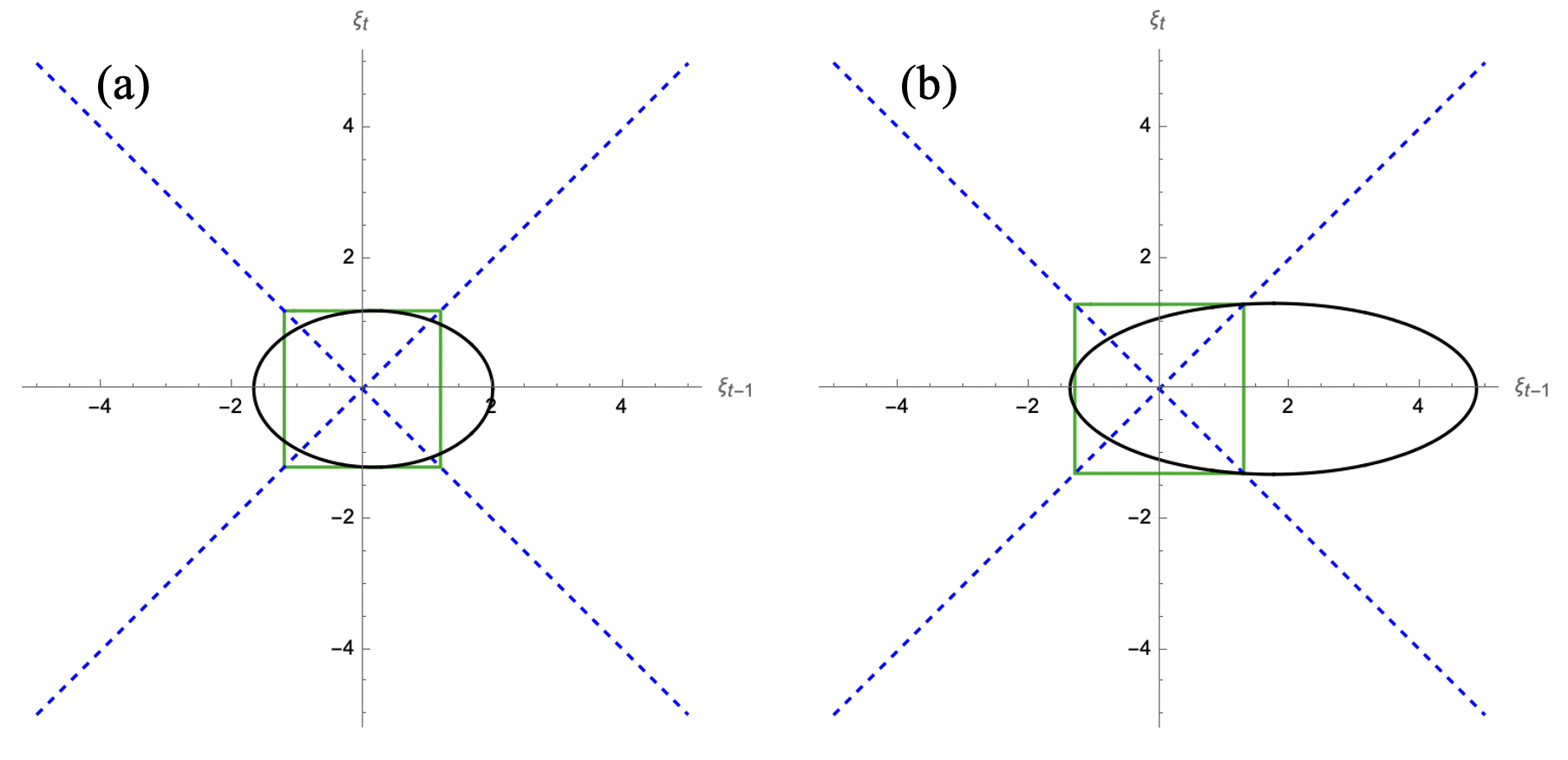}{Elliptical curves for parameters (a) $(r,c)=(0.1,0.3)$ and (b) $(r,c)=(0.525,0.15)$. The box $B^2$ is shown in green, and the diagonals $\xi_t=\pm\xi_{t-1}$ are dashed blue.}{EllipticalExamples}{0.75}
%%%%%%%%

% \[
% \cC_1=\frac{1}{6} (2 - r^2) + \frac{(1 + i \sqrt{3}) (-1024 - 512 r^2 - 64 r^4)}{3072 \cC} - \frac{1}{48} (1 - i \sqrt{3}) \cC,
% \]
% \[
% \cC_2(r)=\frac{1}{6} (2 - r^2) + \frac{(1 - i \sqrt{3}) (-1024 - 512 r^2 - 64 r^4)}{3072 \cC} - \frac{1}{48} (1 + i \sqrt{3}) \cC,
% \]
% for
% \[
% \cC=(-64 + 384 r^2 + 204 r^4 + 26 r^6
%     + 3 \sqrt{3}\sqrt{-2048 r^2 + 4352 r^4 + 5632 r^6 + 2272 r^8 + 392 r^{10} + 25 r^{12}})^{1/3}
% \]

% The roots are given by
% \[
% \cC_1(r)=\frac{1}{48} \left( -8 ( r^2-2 ) - \frac{(
%    i (-i + \sqrt{3}) (r^2 + 4 )^2)}{\cC(r)} + 
%    i (i + \sqrt{3}) \cC(r)\right),
% \]
% \[
% \cC_2(r)=\frac{1}{48} \left( -8 (r^2-2 ) + \frac{(
%    i (i + \sqrt{3}) (r^2+ 4 )^2)}{\cC(r)} - (1 + i \sqrt{3}) \cC(r)\right),
% \]
% for constant
% \beq{EllipseConstant}
% \cC(r)=(26 r^6+ 204 r^4+ 384 r^2 -64+ 3 \sqrt{3} \sqrt{r^2 ( r^2 +4 )^4 (25 r^2-8 )})^{1/3}.
% \eeq

%%%%%%%%%%
%%%%% Hyperbolic
%%%%%%%%%%%
\subsection{Hyperbolic Case}\label{sec:Hyperbola}

For the hyperbolic case $\Delta < 0$. Since $a=1-c$ this implies that $c<0$ or $c>1$.
Supposing first that
\beq{HyperbolaRegion}
c< -\tfrac{r^2}{4},
\eeq
then the branches of the hyperbola are graphs over $\xi_{t-1}$, see the sketch in \Fig{HyperbolicExamples}(a). 
In this case, we can take $B = [-\xi_{max},\xi_{max}]$, defined by \Eq{MaxFixPt}. The images of $B$ then satisfy \Eq{MapsInto}.
The implication is that
\[
%   \left\{(r,c) : |r| < \tfrac{1}{\sqrt{2}} ,\, c <-\tfrac{r^2}{4} \right\} \subset \cR_\cA^+.
   \left\{(r,c) :  c <-\tfrac{r^2}{4} \right\} \subset \cR_\cA^+.
\]
This region is pictured in green in \Fig{RC_PlaneAIExist}(c) and (d).

%we assume $B$ is an interval containing the origin
%then The branches $f_\pm$ could potentially be discontinuous, which would clearly result in the %range \textit{not} being a subset of the domain. %we can have holes in range but NOT in domain
%Thus, first requiring continuity gives the condition
%\[
%    \frac{-c\xi_{t-1}^2+r\xi_{t-1}+1}{a}>0.
%\]
%The leading coefficient of the parabola on the left hand side must be positive, giving $c<0$. %The roots of the parabola, given by
%\[
%\frac{-r\pm\sqrt{r^2+4c}}{-2c},
%\]
%must be imaginary so that the above condition is satisfied, resulting in the region

In addition, note that the asymptotes of the hyperbola are the lines
\beq{Asymptotes}
    \xi_t = \pm \sqrt{\frac{c}{c-1}} \left(\xi_{t-1} - \frac{r}{2c}\right).
\eeq
Consequently when $c < 0$ the magnitude of these slopes is less than one. 
Moreover, since \Eq{HyperbolaRegion} implies that the branches are graphs over $\xi_{t-1}$ these slopes bound those of $f_\pm$ on $B$.
Thus, by \Lem{AIGeneralContraction} for $n=1$, there is a one-to-one correspondence between  AI states  and symbol sequences whenever
\Eq{HyperbolaRegion} is satisfied. 

%%%%%%
\InsertFig{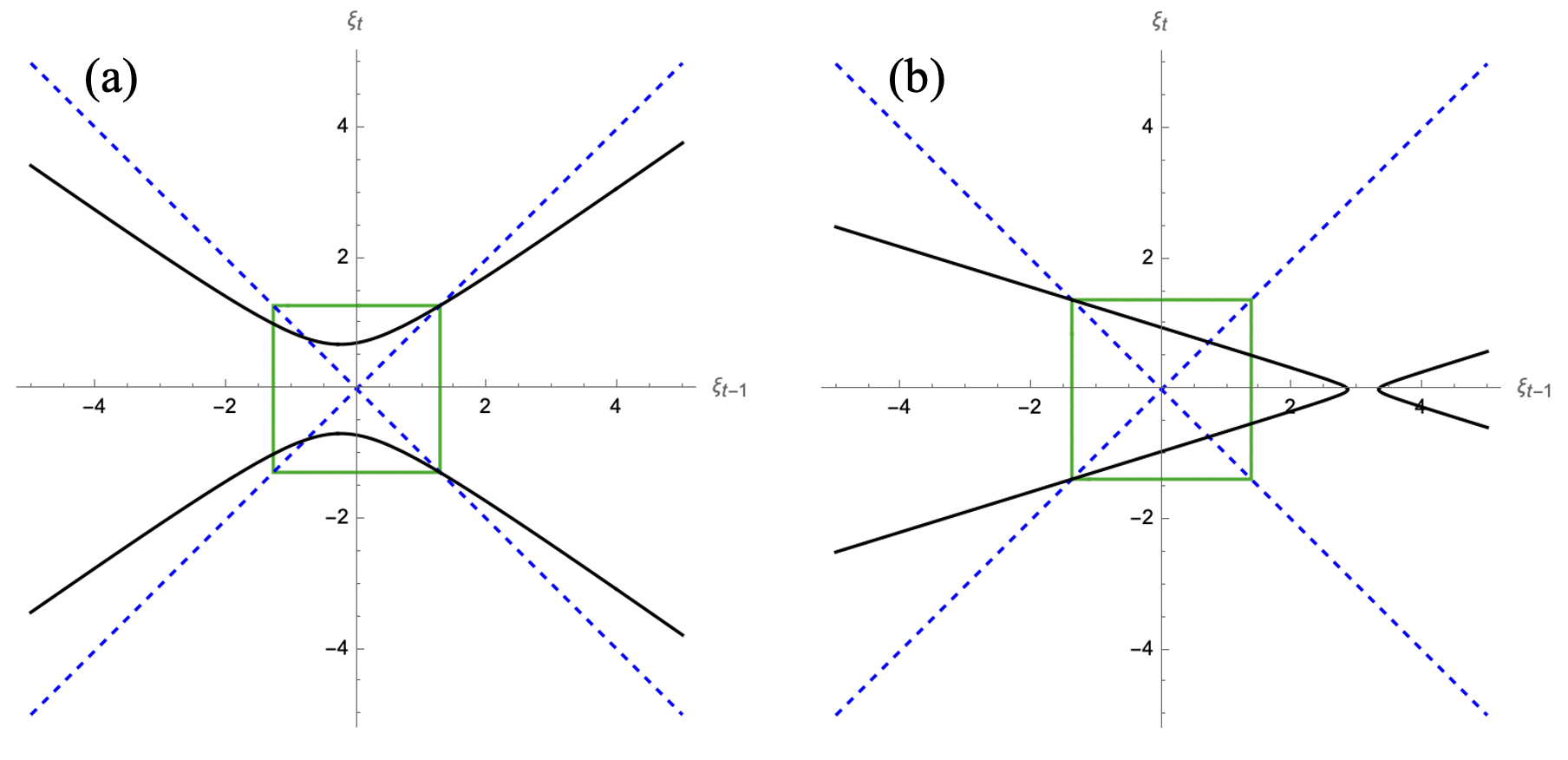}{
Hyperbolic curves for parameters
(a) $(r,c)=(0.5,-1)$,
%When $x_+ > r/2c$ the fixed point is beyond the center 
(b) $(r,c)=(-0.65,-0.105)$. The box $B^2$ is shown in green, and the diagonals $\xi_t=\pm\xi_{t-1}$ are dashed blue.}{HyperbolicExamples}{0.75}
%%%%%%%

When $ -r^4/4 < c < 0$, the branches of hyperbola are graphs over $\xi_t$, as sketched in  \Fig{HyperbolicExamples}(b).
In this case, the domain omits the open interval 
\[
  \left(\frac{1}{2c}(r-\sqrt{r^2+4c}),\, \frac{1}{2c}(r+\sqrt{r^2+4c})\right),
\]
between the vertices. This set must be disjoint from $B$ in order that the maps
$f_\pm$ be well-defined. The vertices are outside $[-\xi_{max},\xi_{max}]$ under
the same condition \Eq{OutlyingCondition} found in the elliptic case.
This implies that
\[
    \left\{ (r,c) : \tfrac{1}{\sqrt{2}} \leq |r| \leq \tfrac{2}{\sqrt{3}} ,\,
    -\tfrac{r^2}{4}<c<1 + |r|(|r|-\sqrt{r^2+4}) \right\} \subset \cR_\cA^+.
\]
This region is pictured in purple for $c<0$ in \Fig{RC_PlaneAIExist}(a) and (d).

Note that condition \Eq{OutlyingCondition} is never achieved when $c>1$; a subset of this case becomes $\cR_A^-$ and is treated in \App{BWMapCalcs}.

%The union of the parabolic-like elliptic region and parabolic-like hyperbolic region is where the %line segment for the parabolic case $c=0, |r|<1/\sqrt{2}$ in black, divides the regions and is %included in $\cR_\cA^+$.

%%%%%%%%%%
%%%%% Strong Contraction
%%%%%%%%%%%
\section{Strongly Contracting Case}\label{sec:SCcase}

The map \Eq{Q3DMap} is volume preserving when $\delta = 1$, and projects to a two-dimensional \hen map in $(x,y)$  when $\delta = 0$.\footnote
{For example, the classic \hen attractor is found when $(\alpha,\sigma, \delta)=(-1.4,-0.3,0)$, or equivalently $(\eps, r) \approx (0.84515,-0.2535)$
\cite[Appendix] {Hampton22}.}
In this section,  we study orbits for a ``strongly contracting case'', setting
\begin{enumerate}
   \myitem[(SC)]\label{SC} Strongly Contracting:  $(a,c,\delta)=(1,0,0.05)$.
\end{enumerate}
We studied this case previously in \cite{Hampton22} for $r=0$;
this corresponds to the AI limit $\alpha \to -\infty$, with $\sigma$ finite.  
There we found a number of periodic attractors, many of which undergo period-doubling bifurcations as $\alpha$ decreases.
 We also found chaotic attractors with a 3D horseshoe-like structure, reminiscent of the \hen attractor.

Here we generalize these results to look at the effect of $r \neq 0$. 
Note that since  $a = 1$ and $c = 0$, the quadratic curve at the AI limit is a parabola. When $|r| < 2^{-1/2}$, the results of \Sec{Parabolic} imply that there is a one-to-one correspondence between symbol sequences and AI states for this case.

%In our previous work, we continued AI states for randomly generated symbol sequences. Here, we are more intentional and continue strategically constructed 
%AI states towards the attractors of case \ref{SC}. 

Characteristics of the attractors for \ref{SC} over a region in $(\alpha,r)$ are shown in \Fig{SCTongue}.
These are found by iterating the difference equation \Eq{VPDiff}, 
upon setting $\sigma = r \eps^{-1} = r\sqrt{-\alpha}$ under the assumption that $\alpha < 0$.
%\beq{Q3DMapWithR}
%    x_{t+1}=\delta x_{t-2} + \alpha -r\sqrt{-\alpha}x_{t-1}+ax_t^2 + cx_{t-1}^2,
%\eeq
%assuming $\alpha\leq0$.
For each $(\alpha,r)$ on a $1500^2$ grid,  we choose the initial point 
$(x_0,x_{-1},x_{-2}) = (x_-,x_-,x_-) + (0.001,0,0)$, where
\[
    x_-=\tfrac{1}{2}\left(1+\sigma-\delta-\sqrt{(1+\sigma-\delta)^2 -4\alpha} \right),
\]
is a fixed point of \Eq{VPDiff} This point is iterated forward $T = 5000$ times to eliminate transients. An orbit is declared divergent if, for some $t \le T$, $|x_t| > \kappa_{max}$\footnote
%%%%
{
In \cite{Hampton22b}, it was proved that all bounded orbits lie in the set 
$|x| <\kappa$ for a given $\kappa>0$ depending upon the map parameters.
%=\tfrac12 \left(|r\sqrt{-\alpha}|+\delta+1 + \sqrt{(|r\sqrt{-\alpha}|+\delta+1)^2+4|\alpha|} \right)$.
We take $\kappa_{max}$ to be the maximum value of $\kappa$ over the parameters studied.
}
%%%%
, for $\kappa_{max}=3.26724$ in \Fig{SCTongue}(a), and $\kappa_{max}=2.28343$ in \Fig{SCTongue}(b).
Such divergent cases are colored white in \Fig{SCTongue}.
If the orbit remains bounded, we detect low periods by iterating up to $90$ more steps, 
checking for a close return: the approximate period, $p$, is the smallest time for which
\[
    \|x_{T+p} - x_T\| < 10^{-4}.
\]
Orbits with these periods are colored according to the color map shown in \Fig{SCTongue}. For the remaining parameter values
we label the orbit as regular (black) or chaotic (grey) by computing the approximate maximal Lyapunov exponent as described in \cite{Hampton22b}. 
Thus, the colors of \Fig{SCTongue} indicate the type of attractor expected for each $(\alpha,r)$. 

%%%%%%
\InsertFig{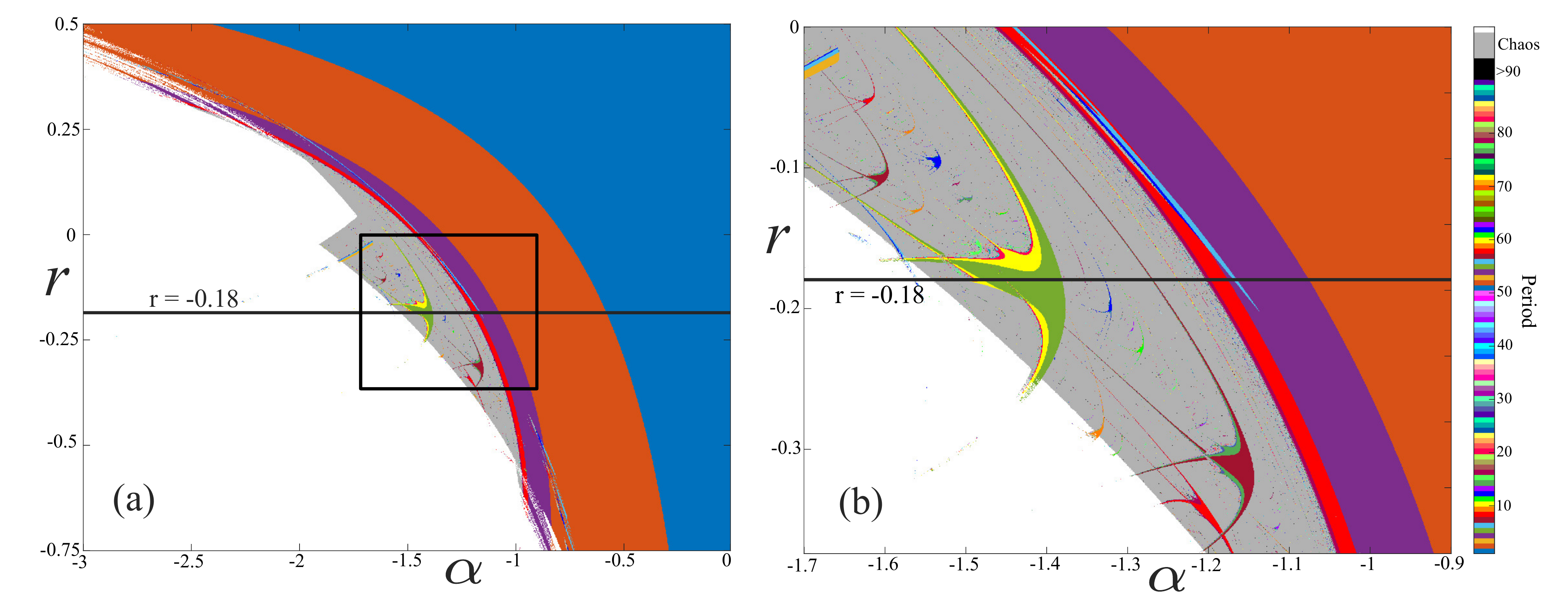}
{(a) Bounded, periodic and chaotic orbits for a strongly contracting case of map \Eq{Q3DMap} with parameters \ref{SC}.
The color scale indicates the period with chaotic orbits colored gray and unbounded orbits white. If the orbit is
bounded, not chaotic, but not identified as having period at most $90$, the point is black. Also pictured is a line segment at $r=-0.18$. (b) An enlargement of the boxed region in (a) around the period-five shrimp.}{SCTongue}{1.1}
%Input for Henon: [Henon_alphavect, Henon_rvect, Henon_period, Henon_MaxLyapunovExp_vector, Henon_MLE_convergence_vector] = TonguesAndMLEs3D_RC(0,0.05,.0003,-3,0,-.75,.5,500,500 );
%%%%%

Note that since we study only a single initial condition, we cannot rule out the appearance of multiple attractors, nor the existence of attractors that might occur for other initial points in the ``undbounded'' region of the figure. Indeed, we show in \Fig{SCTongueDifferentIC} the same parameter range as \Fig{SCTongue}(a), but now choosing the initial point $(x_0,x_{-1},x_{-2}) = (0.0125839,0.677585,-1.25765)$. Note that there is a striking absence of the stable fixed point and its doubling sequence for $r > 0$ and $\alpha$ small as compared to \Fig{SCTongue}. Moreover there is a new stable region (gold) corresponding to a period-$3$ attractor along with a small doubling-cascade to periods 6 and 12.

\InsertFig{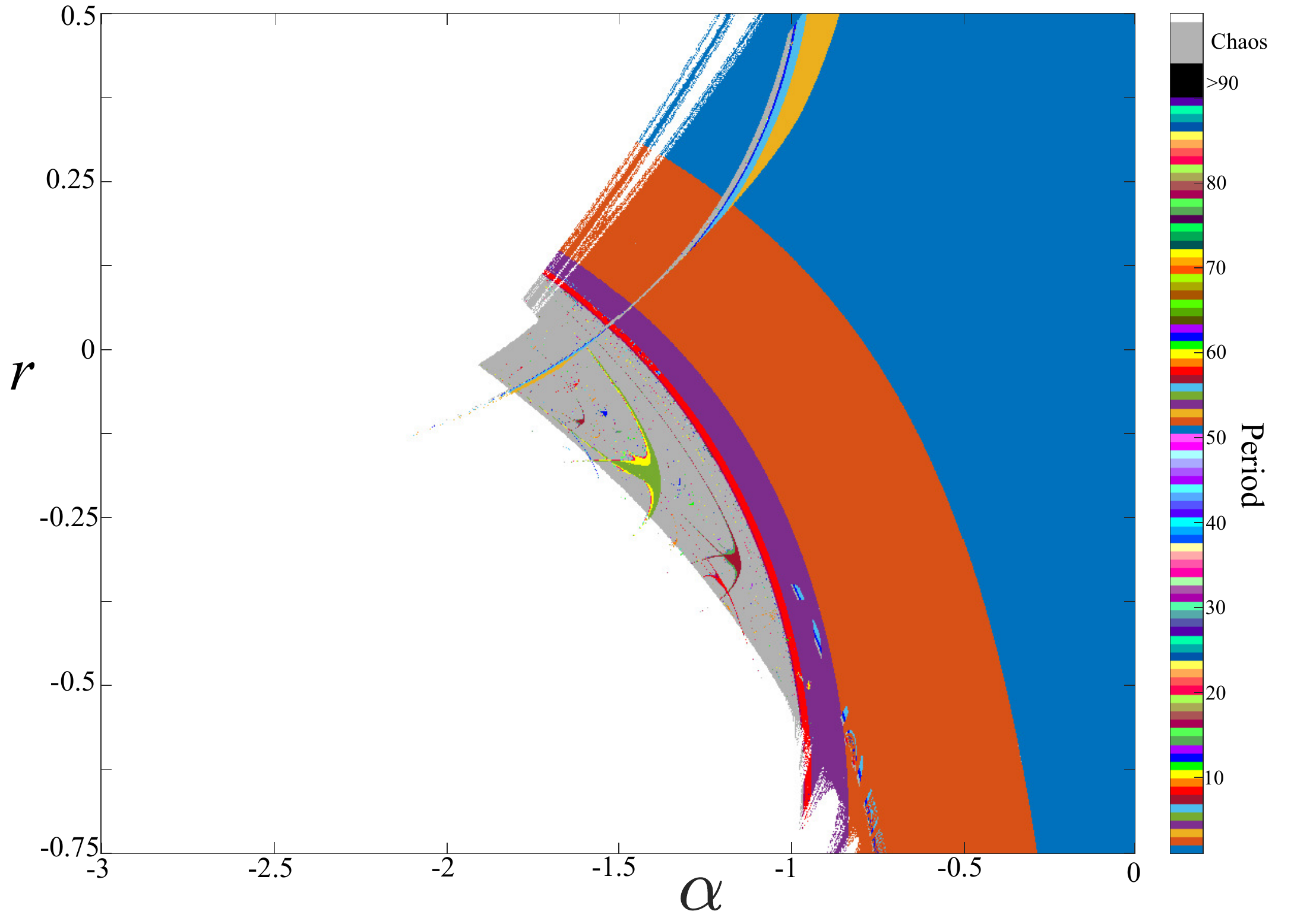}{Bounded, periodic and chaotic orbits for case \ref{SC} over the same parameter range as \Fig{SCTongue}(a), but using a different initial condition. The new gold, period-3 attracting region not seen in \Fig{SCTongue}(a) shows the possibility of multiple attractors.}{SCTongueDifferentIC}{0.55}

Figure \ref{fig:SCTongue}(a) shows the range $(\alpha, r) \in [-3,0] \times [-0.75,0.5]$. The fixed point $x_-$ is stable in the ``strong blue'' region, and undergoes a doubling bifurcation along the curve
% \[
%     \alpha = \frac{2(c-a)((r\sqrt{-\alpha}+1)^2-\delta^2)-(r\sqrt{-\alpha}+\delta+1)^2}{4(c-a)^2},
%    % \alpha &= \frac{r\sqrt{-\alpha}+\delta+1}{2(c-a)}(2x_{SN}-x_{PD}) \\
%    % x_{SN} &= \tfrac{1}{2}(r\sqrt{-\alpha}-\delta+1) \\
%    % x_{PD} &= \frac{r\sqrt{-\alpha}+\delta+1}{2(c-a)}
% \]
%Here is a form when $a=1,c=0$:
\beq{DoublingCurve}
    (3r^{2}-4)^{2} \alpha^{2}+2 \left((5\delta^{2} +6\delta + 9)r^{2}-4 \delta^{2}+8 \delta +12\right) \alpha +\left(\delta +1\right)^{2} \left(\delta -3\right)^{2} = 0,
\eeq
to become a stable period-two orbit (vivid orange). Subsequent doublings as $\alpha$ decreases create period-four (magenta) and period-eight (red) orbits.

An enlargement of the boxed region in \Fig{SCTongue}(a) is shown in \Fig{SCTongue}(b) for $(\alpha,r) \in [-1.7,-0.9] \times [-0.375,0]$.
Prominent features in this region are resonant ``shrimps'' including the period-five (dark green) and period-seven (dark red) cases. A shrimp is a codimension two structure much studied in two-parameter families of one and two-dimensional maps \cite{Gallas94,MacKay87f,Facanha13}. Within a period-$n$ shrimp, there is an attracting period-$n$ orbit and a partnered period-$n$ saddle. The ``head'' of the shrimp corresponds to a pair of curves of saddle-node bifurcations and the ``tail'' to sequences of period-doubling bifurcations. The endoskeleton of the shrimp is near the curve where the trace of the Jacobian is zero \cite{Facanha13}. As can be seen in the figure, shrimps swim in a ``sea of chaos''. 

% Color Names (following https://www.colorhexa.com/edb134
% \begin{align*}
% \begin{array}{ccccl}
% p  & r & g & b & name\\
% 1  &0     &114  &189  & strong blue\\
% 2  &217   &83   &25   &vivid orange\\
% 3  &237   &177  &32   &bright orange (should be 52 in B)\\
% 4  &126   &47   &142  &magenta\\
% 5  &119   &172  &48   &dark green\\
% 6  &77    &190  &238  &soft blue\\
% 7  &162   &19   &47   &brown (should be 42 G) (dark red)\\
% 8  &256   & 0   & 0   &pure red
% \end{array}
% \end{align*}

To study the bifurcations in more detail, we set $r=-0.18$, seen as a line segment in \Fig{SCTongue}. This segment enters the region of bounded orbits at $\alpha=-1.541$, crosses the dark green, period-five shrimp when $-1.480\leq\alpha\leq-1.381$, and enters the doubling cascade of the fixed point at $\alpha=-1.2031$. In the following subsections, we will continue periodic orbits from $\eps = 0$ along this line using
the numerical continuation algorithm discussed in \App{NumericalContinuation}.
%recognizing patterns in the bifurcations that align with a conjecture proposed in \cite{Hampton22} and \cite{Sterling98}. 
%We then continue a number heteroclinic orbits from the AI limit, using a construction from \cite{Sterling99} that exploit the period-5 shrimp and period-doubling cascade of the fixed point. 
%Lastly, we continue AI states that correspond to a specific chaotic attractor found at $\alpha=-1.25$.

%%%%%%%%%%
%%%%% Period-doubling
%%%%%%%%%%%
\subsection{Low-Period Orbits and the Period-Doubling Cascade}\label{sec:LowPeriod}

There are 23 possible periodic symbol sequences with periods $p \le 6$. 
Note that by \Sec{Parabolic}, for the case \ref{SC}, when $|r| < 2^{-1/2}$
there is a one-to-one correspondence between symbol sequences and AI states.
The AI states can be easily found by iteration from an arbitrary point in the interval $B$, 
since the maps $f_s$ \Eq{AIMap} are contracting in this case.
In this section we will continue each of these AI states for $r = -0.18$.

We expect that at least one period-$n$ AI state will continue to the $\alpha$ range where there is a stable period-$n$ orbit  seen in \Fig{SCTongue}. This is in alignment with the ``no-bubbles'' conjecture proposed in \cite{Sterling99}: every orbit of the map is continuously connected to the AI limit. 
%For example, we'd expect at least one period-5 AI state to continue to at least the head of the period-5 shrimp at $\alpha=-1.381$.
Since the attractors in \Fig{SCTongue} are found with a specific initial condition, it is certainly possible that AI states continue beyond the $\alpha$ range in the figure: 
there could be multiple attractors with basins that may or may not contain our chosen initial point. 

The results of the continuation for orbits up to $p=5$ are shown in a bifurcation diagram, projected onto the $\xi_t$-axis in \Fig{BifDiagPds1thru5}(a), with each periodic orbit labelled as its symbol sequence, given in the legend of the figure.
Note that the fixed points $\{\pm\}^\infty$ extend for the entire $\eps$ range shown in the diagram, but all of the higher-period orbits are destroyed by $\eps = 1.4$; the final observed bifurcation is the reverse period-doubling that
destroys the period-two orbit $\{-+\}^\infty$ when it collides with $\{-\}^\infty$. Notice that the AI states themselves, i.e., the orbits at $\eps=0$ all lie on a set of finite points, or some Cantor-like set, which is in agreement with the intuition of \Con{RegionConvergence} discussed above.

%%%%%%%%
\InsertFigTwo{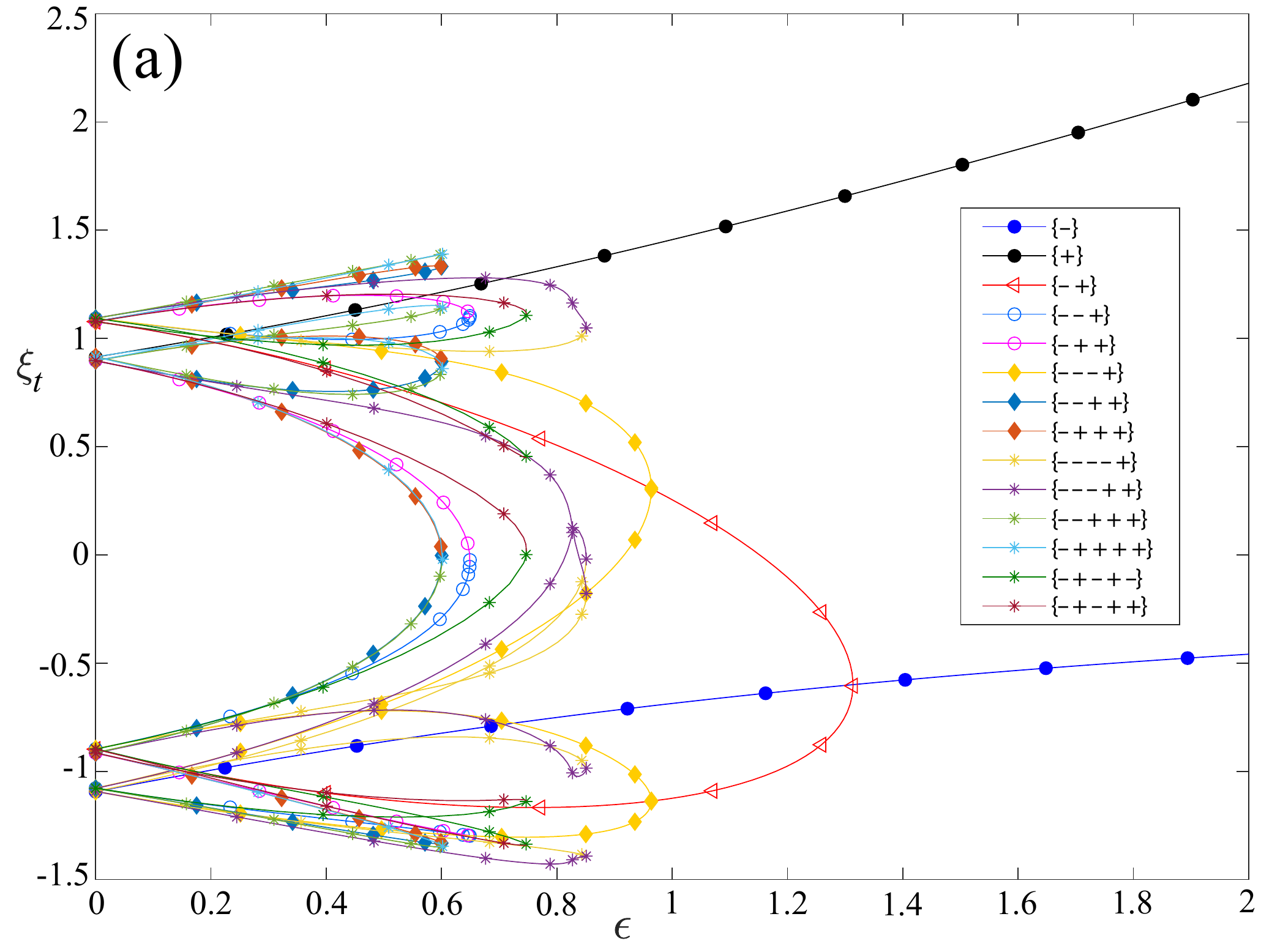}{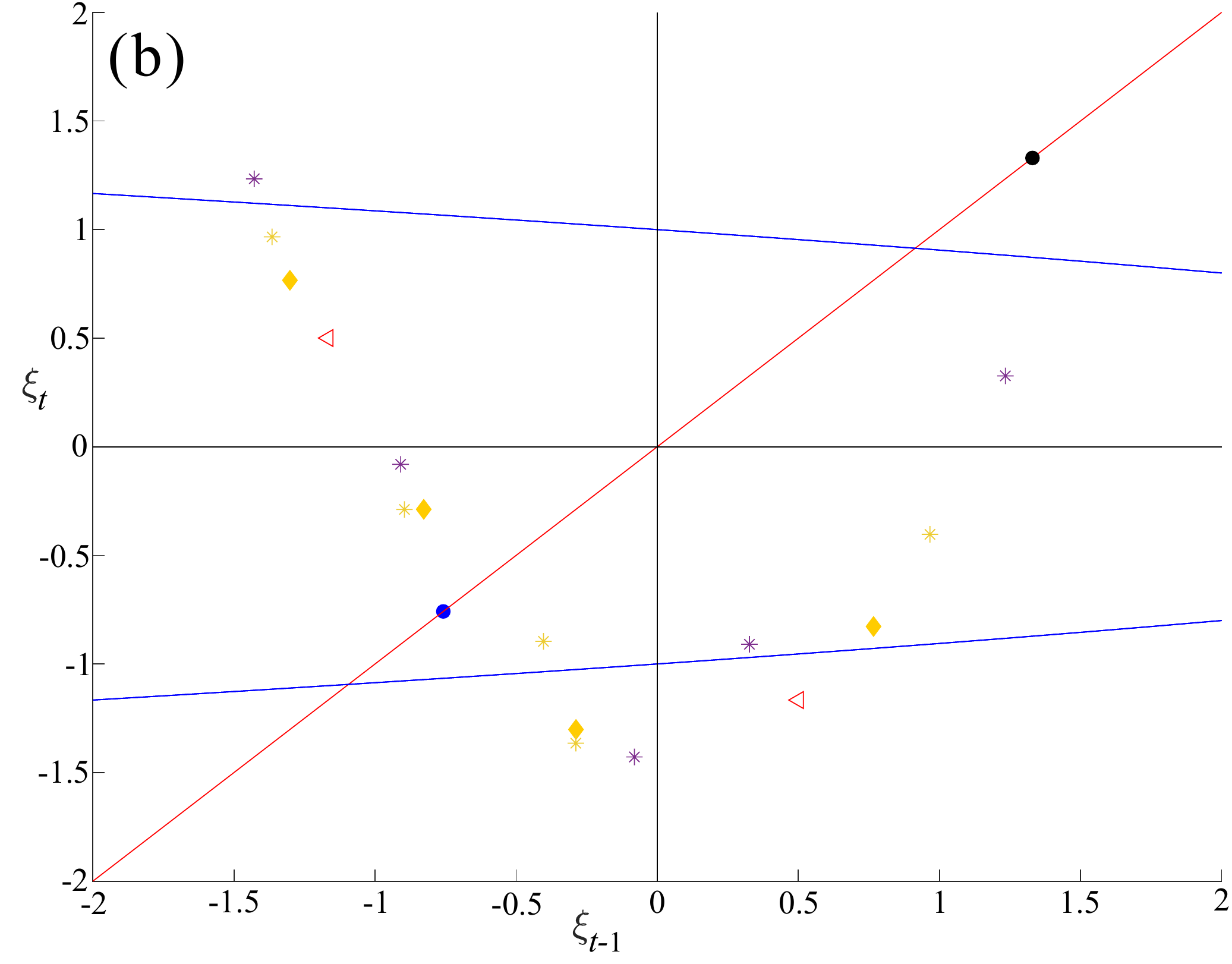}{(a) Bifurcation diagram for low-period orbits up to period-five for case \ref{SC} along $r=-0.18$. (b) 2D projection of orbits from (a) for the continuation step  closest to $\eps=0.8$. Symbols for the orbits are the same as in (a). Included is the AI limit curve \Eq{ReducedAILimit} (blue), which is a parabola for this case, and the diagonal (red) for reference.}
{BifDiagPds1thru5}{0.5}
%%%%%%%%%%%

%%%%%%%%%%%%%%
\begin{table}[h!t]
\centering
\begin{tabular}{c|c|c|c|c}
%     \hline
    Parent & Type & Child & $\eps$ & $\alpha$ \\
    \hline
    $\{-\}^\infty$ & pd
    %%%
    & $\{-+\}^\infty$  & 1.3136 & -0.5795  \\ 
    \hline
     $\{-+\}^\infty$ & pd & $\{---+\}^\infty$  & 0.9639 & -1.0764   \\ 
    \hline
    $\{-++\}^\infty$ & pd & $\{--++-+\}^\infty$  & 0.6478 & -2.3843  \\
    \hline
    $\{---+\}^\infty$ & pd & $\{---+-+-+\}^\infty$  & 0.9220 & -1.1763   \\ 
    \hline
     & sn & $\{-\pm+\}^\infty$   & 0.6492 & -2.3723  \\
    \hline
     & sn & $\{-\pm++\}^\infty$ & 0.6002 & -2.7762  \\ 
    \hline
     & sn & ${\{---\pm+\}^\infty}$ & 0.8510& -1.3808  \\  
    \hline
     & sn & ${\{-+-+\pm\}^\infty}$  & 0.7473& -1.7908   \\
   \hline
     & sn & $\{-\pm+++\}^\infty$  & 0.6017& -2.7622   \\
    \hline
      & sn & $\{---\pm-+\}^\infty$ & 0.9266& -1.1648   \\
    \hline 
     & sn & $\{---\pm++\}^\infty$ &  0.8007& -1.5597  \\
     \hline
     & sn & $\{-+-++\pm\}^\infty$ & 0.6889& -2.1070   \\
      \hline
     & sn & $\{-\pm++++\}^\infty$ & 0.6005& -2.7734   \\

\end{tabular}
\caption{Parameters $\eps$ and $\alpha=-\eps^{-2}$ for period-doubling (pd) and saddle-node (sn) bifurcations for all orbits up to period six for
the case \ref{SC} with $r=-0.18$. Orbits are identified by their symbol sequences in the first and third columns. For saddle-node bifurcations, the symbol sequences of the two colliding orbits are listed together: the $\pm$ indicates the single symbol that differs.}
\label{tbl:LowPdBifTab}
\end{table}
%%%%%%%%%%%%%%

The period-doubling and saddled-node bifurcations that are found by this process are summarized in \Tbl{LowPdBifTab}.
The fourth column of the table gives a $5$-digit estimate of the bifurcation value, from the continuation 
algorithm. Analytically, using \Eq{DoublingCurve} for $r = -0.18$ 
the fixed point doubles at $\alpha_{PD}=-0.579494815477836$. The
computations continue to  $\alpha=-0.5794956$, which differs by $8(10)^{-7}$.
%$\alpha=-0.579495636804485$
Since we have not used a bifurcation detection criterion in the continuation algorithm,
we do \textit{not} expect high accuracy for these values. 
Moreover, detecting bifurcations using multipliers can be problematic for $\delta = 0.05$:
since the product of the multipliers of the
linearization of \Eq{Q3DMap} is its Jacobian determinant,
then for a period-$n$ orbit, $\lambda_1\lambda_2\lambda_3=\delta^n$.
Implementation of such a detection criterion to compute $\lambda = 1$ or $-1$ requires high precision computations as $n$ grows.
%Thus computing the multipliers is increasingly problematic for larger $n$. 
% Thus, we use the number of rounded matching digits (i.e., 5) the continuation algorithm reaches for the 
%doubling of the fixed point as our standard. Future research will investigate higher precision numerical 
%continuation methods.

Note that in each case shown in \Tbl{LowPdBifTab}, the codimension-one bifurcations occur between orbits with exactly one differing symbol. This agrees with our previous conjecture in \cite{Hampton22}.

Both of the period-three orbits are destroyed by saddle-node bifurcations below the onset of the bounded attractors of \Fig{SCTongue}, at $(\alpha,r)= (-1.541,-0.18)$.  This agrees with the absence of period-three attractors along the line $r = -0.18$. Two of the three period-four orbits are similarly destroyed in a saddle-node, 
at $\alpha = -2.7762$, before the onset of bounded attractors.
The remaining period-four orbit, $\{---+\}^\infty$, becomes stable in the magenta region of \Fig{SCTongue}, and is destroyed by a (reverse) period-doubling of the $\{ -+ \}^\infty$ orbit at $\alpha = -1.0764$.
Only two of the period-five orbits continue into the bounded region, $\{---\pm+\}^\infty$; these are the orbits that form
the period-five shimp in \Fig{SCTongue}. The attracting orbit has the sequence $\{---++\}^\infty$; the other, $\{----+\}^\infty$, is a saddle. These orbits collide in a saddle-node bifurcation just as the line segment exits the shrimp at $\alpha=-1.3808$. 

The pair $\{---\pm-+\}^\infty$ of period-six orbits continues into the bounded orbit region and enters the period-doubling cascade of the fixed point. In the the thin, soft blue strip seen in \Fig{SCTongue}, the orbit $\{---+-+\}^\infty$ is stable. In this region there is also an attracting period four orbit: this is a case of multiple attractors. 

As \Fig{SCTongue} indicates, the fixed point $\{-\}^\infty$ undergoes a period doubling cascade as $\alpha$ decreases; 
the first doublings correspond to the symbol sequences shown in \Tbl{DoublingSequences}.
These have a simple pattern if the sequences are properly ordered (as in \Tbl{DoublingSequences}): to get the sequence of period $2^{n+1}$, simply double that for period-$2^n$ and flip the first sign.
This pattern is related  to that found in \cite{Hao91} for one-dimensional maps. 
This pattern also seems to hold for the doubling of the period-3 orbit, as seen in the third row of \Tbl{LowPdBifTab}, if the period-3 orbit is first written as $(+-+)$. This pattern was verified for a finite number of steps. Below we propose a formal conjecture.

%%%%%%%%%%%%%%
\begin{table}[h!t]
\centering
\begin{tabular}{r|l}
%%%%%%%%%%
Period & Sequence \\
\hline
1 & $\{-\}^\infty$ \\
2 & $\{+(-)\}^\infty$ \\
4 &  $\{--(+-)\}^\infty$ \\
8 &  $\{+-+-(--+-)\}^\infty$\\
 %           $\{---+ (-+)^2 \}$
16 &  $\{--+---+-(+-+---+-)\}^\infty$ \\
  %           $\{(---+(-+)^2)(---+)^2\}$
32 &  $\{+-+---+-+-+---+-(--+---+-+-+---+-)\}^\infty$\\
  %          $ \{((---+)(-+)^2)^3 (---+)^2\}$
\end{tabular}
\caption{Symbol sequences for the period-doubling cascade of the fixed point $\xi_-$.}
\label{tbl:DoublingSequences}
\end{table}
%%%%%%%%%%

\begin{con}\label{con:PdDoubling}
Symbol sequences between period-doubled periodic orbits of the map \Eq{Q3DMap} have exactly one differing symbol. When written appropriately, the sequence of the doubled orbit is simply double that of the original orbit with the first sign flipped.
\end{con}

In \Fig{BifDiagPds1thru5}(b), the six orbits with $p\le 5$ that exist at $\eps =0.8$ are projected onto the $(\xi_{t-1},\xi_{t})$ plane. For these parameters, the line segment $r=-0.18$ has not yet entered the region of bounded attractors. Nevertheless this figure indicates that the low-period orbits appear to trace a horseshoe-like structure reminiscent of the \hen attractor.

%%%%%%%%%%
%%%%% Henon Attrators
%%%%%%%%%%%
\subsection{H\'{e}non-like Attractors}\label{sec:HenonAttractors}

When $r=-0.18$, chaotic attractors are indeed found when 
$\alpha$ is in the gray regions of \Fig{SCTongue}, just outside the period-five shrimp.
One such  attractor with $\alpha=-1.25$ is shown in \Fig{ChaoticAttractors_ContSoln}(a).
Indeed, this is a chaotic attractor since, using the method described in \cite{Hampton22b}, its maximal Lyapunov exponent is $1.30116$. To understand how this attractor develops, we first find a periodic approximation by iterating the point $(-1.3387,-0.2563,-0.9553)$ on the attractor using map \Eq{Q3DMap} until it exhibits a close return---within a distance of $0.005$. The first three such close return times are listed in \Tbl{ChaoticAttractors}.

%%%%%%%%%%%%%%
\begin{table}[h!t]
\centering
\begin{tabular}{r|c|c|c}

     period & return distance & $\eps$ & $\alpha$ \\%  & -0.2535\\
    \hline
    273 & $7.96(10)^{-5}$ & 0.894 & -1.250 \\
    423 & $2.4(10)^{-3}$ & 0.894 & -1.252 \\
    1200 & $4.0(10)^{-5}$ &  0.893 & -1.255 \\

    %pre-shrimp attractor & 312 & 0.00005 & 0.761 (-1.725)  \\ % & 0.845 (-1.401) \\
    
    %post-shrimp attractor & 519 & 0.001 & 0.876 (-1.302) \\% & 0.556 (-3.189) \\
\end{tabular}
\caption{Approximate periods and return distances for orbits on the chaotic attractor for 
$(\alpha,r) = (-1.250,-0.18)$. The last two columns show bifurcation values for orbits
continued from the AI limit using the inferred periodic symbol sequences.}
\label{tbl:ChaoticAttractors}
\end{table}
%%%%%%%%%%%%%%

Using these orbits we can construct an associated symbol sequence under the assumption that
\[
    s_t=\sign{(\xi_t)},
\]
as it would be needed to use \Eq{AIMap} at the AI limit.
%by first solving \Eq{RescaledDifEq} for $\xi_t$, giving us an expression in terms of $s_t$ and $\eps$ for the full 3D map, and then solving for $s_t$, resulting in
% \[
% \xi_t = s_t\sqrt{\frac{- 
%    c \xi_{t-1}^2 + r \xi_{t-1} +1  + \eps(\xi_{t+1}-\delta \xi_{t-2})}{a}}
% \]
% \[
% s_t =\xi_t\sqrt{\frac{a}{- 
%    c \xi_{t-1}^2 + r \xi_{t-1} +1  + \eps(\xi_{t+1}-\delta \xi_{t-2})}}.
% \]
Thus each periodic approximation has a corresponding periodic symbol sequence, which we use to find an AI state. These are then continued away from $\eps = 0$, again using the method described in \App{NumericalContinuation}.
The resulting orbits persist up to the values of $\eps$ shown in column three of \Tbl{ChaoticAttractors}.
These orbits are shown at these maximal $\eps$-values in \Fig{ChaoticAttractors_ContSoln}(b,c,d).
The period-$273$ approximation to the chaotic attractor persists the longest, reaching $\alpha = -1.250$. The two longer periodic approximations do not continue as far, even though the period-$1200$ orbit has a smaller return distance. In \Fig{MoviePanels}, there are six panels that follow the continuation solution over increasing $\eps$ of the period-1200 orbit (black) in $(\xi_{t-1},\xi_t,\xi_{t+1})$-space. Also shown are the attractors (blue) as detected by the algorithm used to create \Fig{SCTongue}. In \Fig{MoviePanels}(a), the period-1200 orbit is at the AI limit; it lies on a Cantor-like set. In \Fig{MoviePanels}(b,c), $\eps$ has not yet reached the region of bounded attractors, and the period-1200 orbit continues to evolve by growing apart and onto a folded structure. Only \Fig{MoviePanels}(d) shows an instance where both the continued period-1200 orbit and an attractor co-exist; they appear to cover the same invariant set. For slightly larger $\eps$,
the period-$1200$ orbit is destroyed and, as seen in \Fig{MoviePanels}(e), the attractor has split into two chaotic bands. Effectively, these  have ``(reverse) merged'' from the chaotic attractor that was seen in panel (d). 
These bands subsequently collapse in a (reverse) period-doubling cascade; the period-two case is shown in panel \Fig{MoviePanels}(f).
%It is not obvious as to why the shortest periodic approximation to the chaotic attractor does the best in regards to persistence.

For reference we give the period-273 symbol sequence:
\begin{align*}
\{&(-(+-)^4-^2(+-)^6-^2(+-)^6-^2(+-)^2-^2(+-)^3-^2(+-)^2-^2(+-)^7 -^2(+-)^9-^2(+-)^2-^2\\
&(+-)^4-^2(+-)^2-^2(+-)^2-^2 (+-)^2-^4+-^3(+-)^2-^7(+-^3)^8(+-)^3-^2(+-)^5-^2\\
&(+-)^2-^2(+-)^3+-^3(+-)^2-^2(+-)^3-^2(+-)^5-^2(+-)^2-^2(+-)^3-^2(+-)^2-^2(+-)^5-
\}
\end{align*}
It is interesting that each subsequence of 
this orbit can be seen to be one of the sequences along the doubling cascade of the fixed point, i.e., from the list in \Tbl{DoublingSequences}. 
For example, in condensed form, the period-four orbit can be represented by either $\{(+-)-^2\}^\infty$ or $\{+-^3\}^\infty$, the period-eight orbit as $\{(+-)^3-^2\}^\infty$, etc. Chaotic attractors are often found after period-doubling cascades, so it seems apt that the symbolic sequence associated with a chaotic attractor is made up of such subsequences. It is interesting to speculate that one could predict the symbol sequence of such a chaotic attractor without a formal calculation of $s_t$.

\InsertFigOneThree{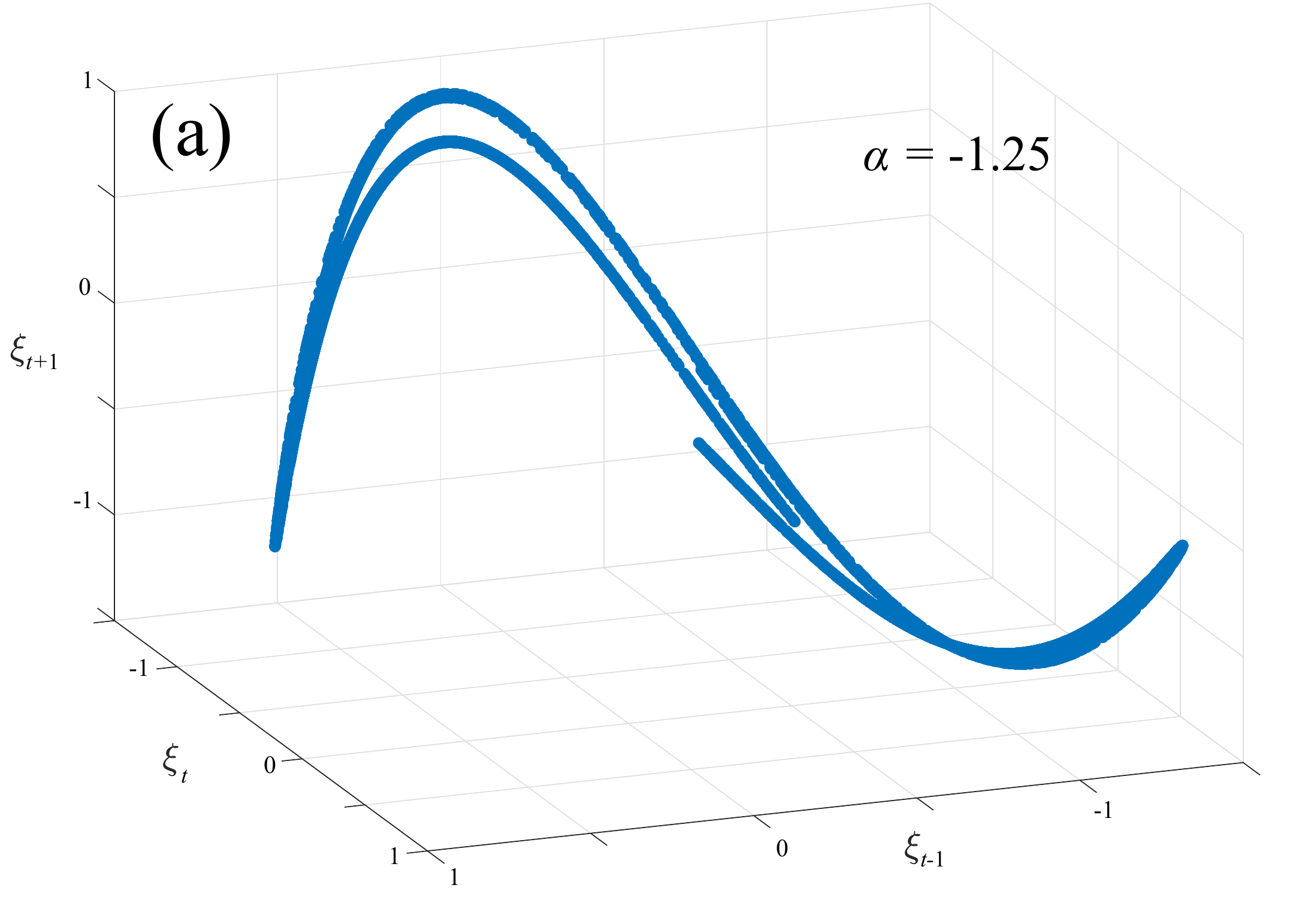}{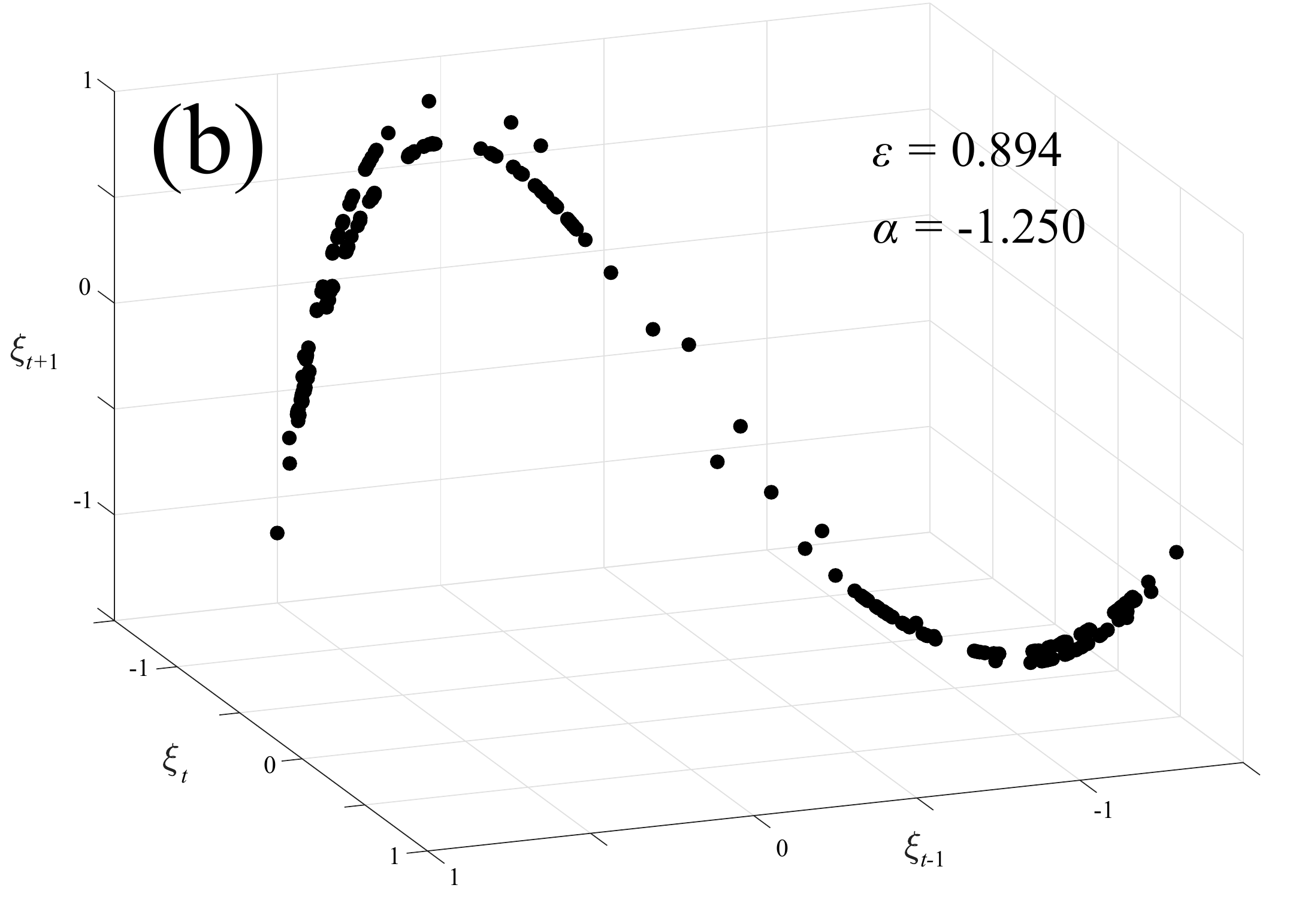}{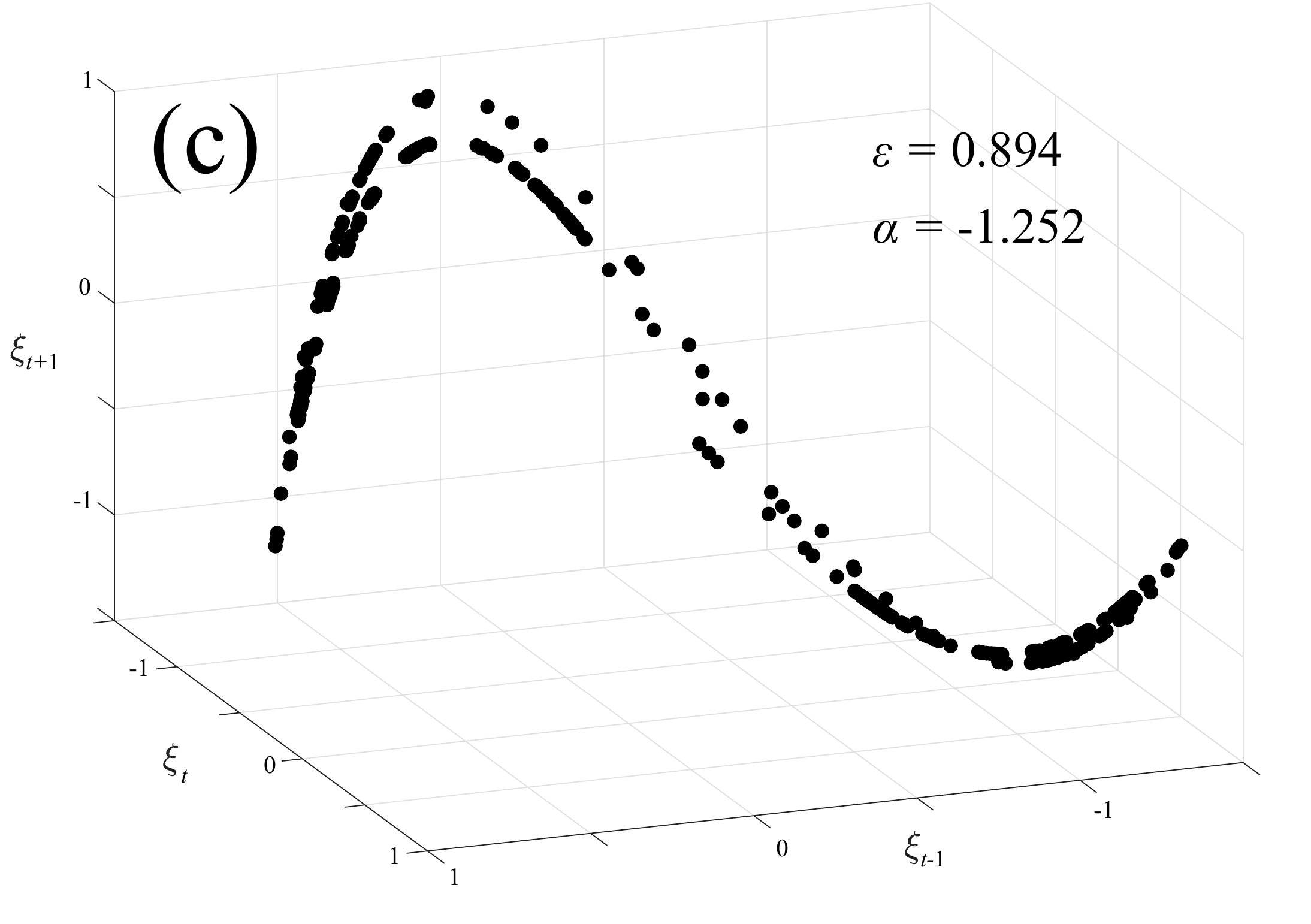}{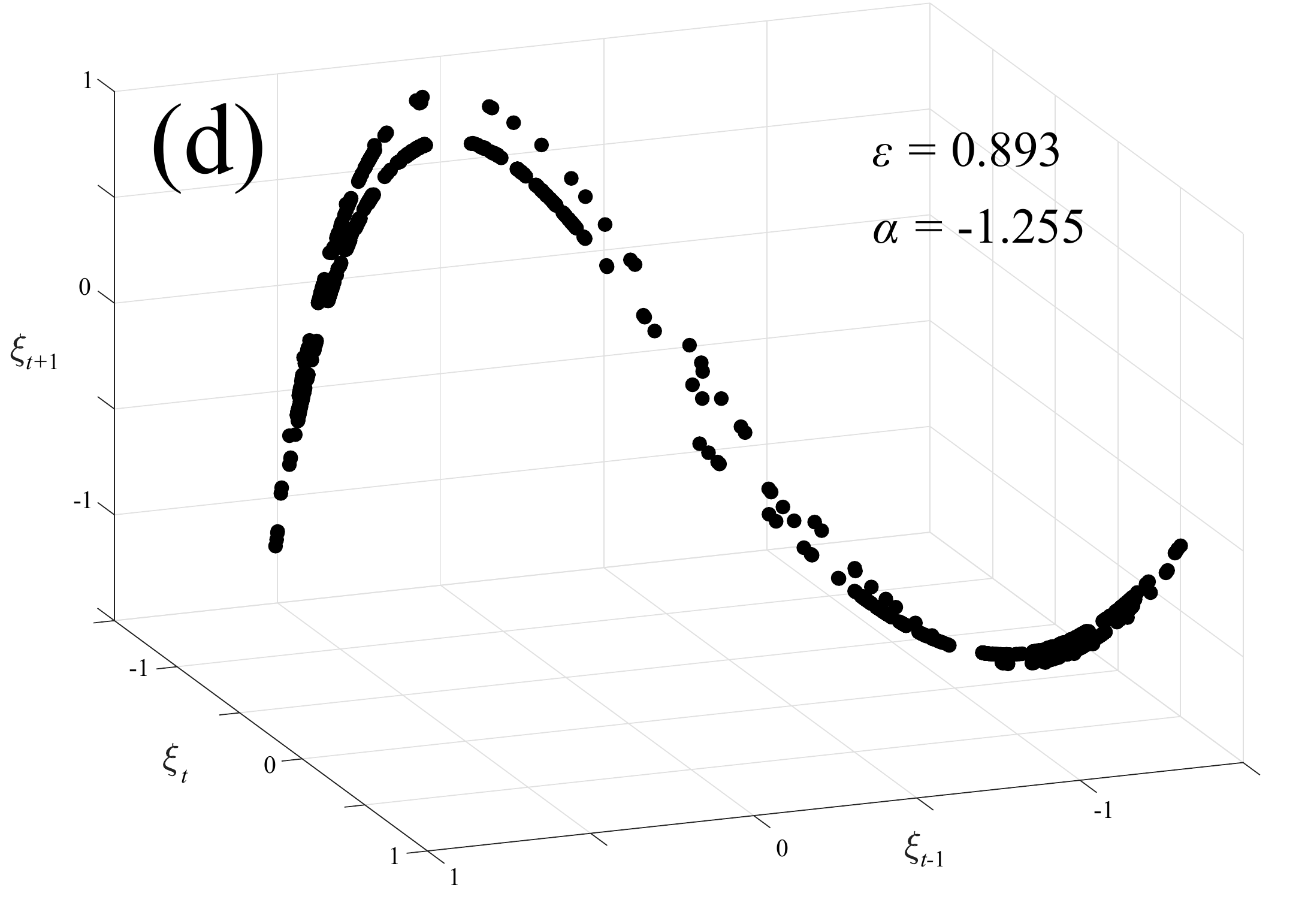}{(a) The chaotic attractor at $(\alpha,r)=(-1.25,-0.18)$ plotted in the rescaled spatial coordinates $\xi=\eps x$. Continuation results for the period (b) 273, (c) 423, and (d) 1200 orbits that are obtained by close returns on the chaotic attractor. These are shown for $\eps$ given in the last column of \Tbl{ChaoticAttractors}.}{ChaoticAttractors_ContSoln}{0.32}{0.6}

\InsertFigSix{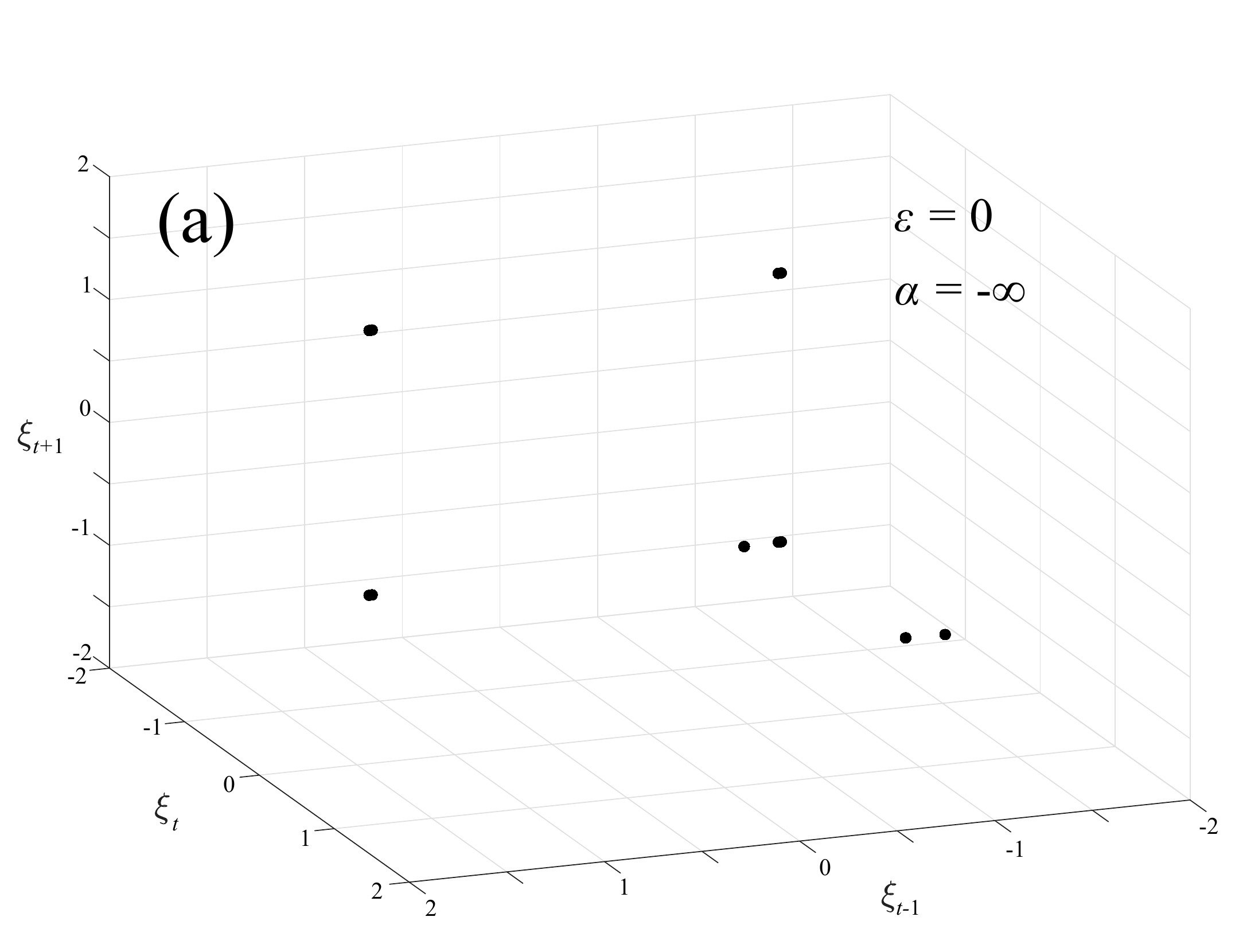}{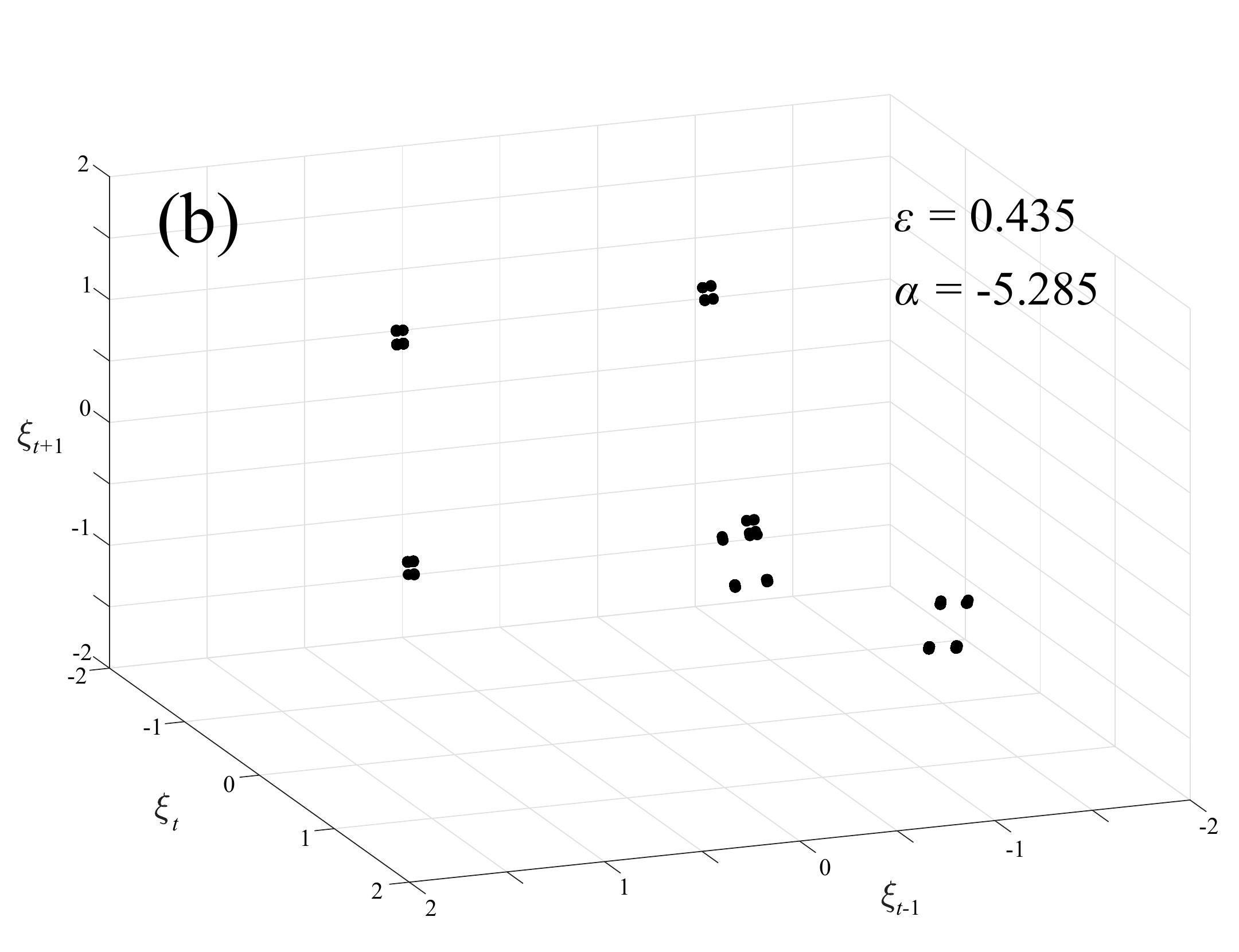}{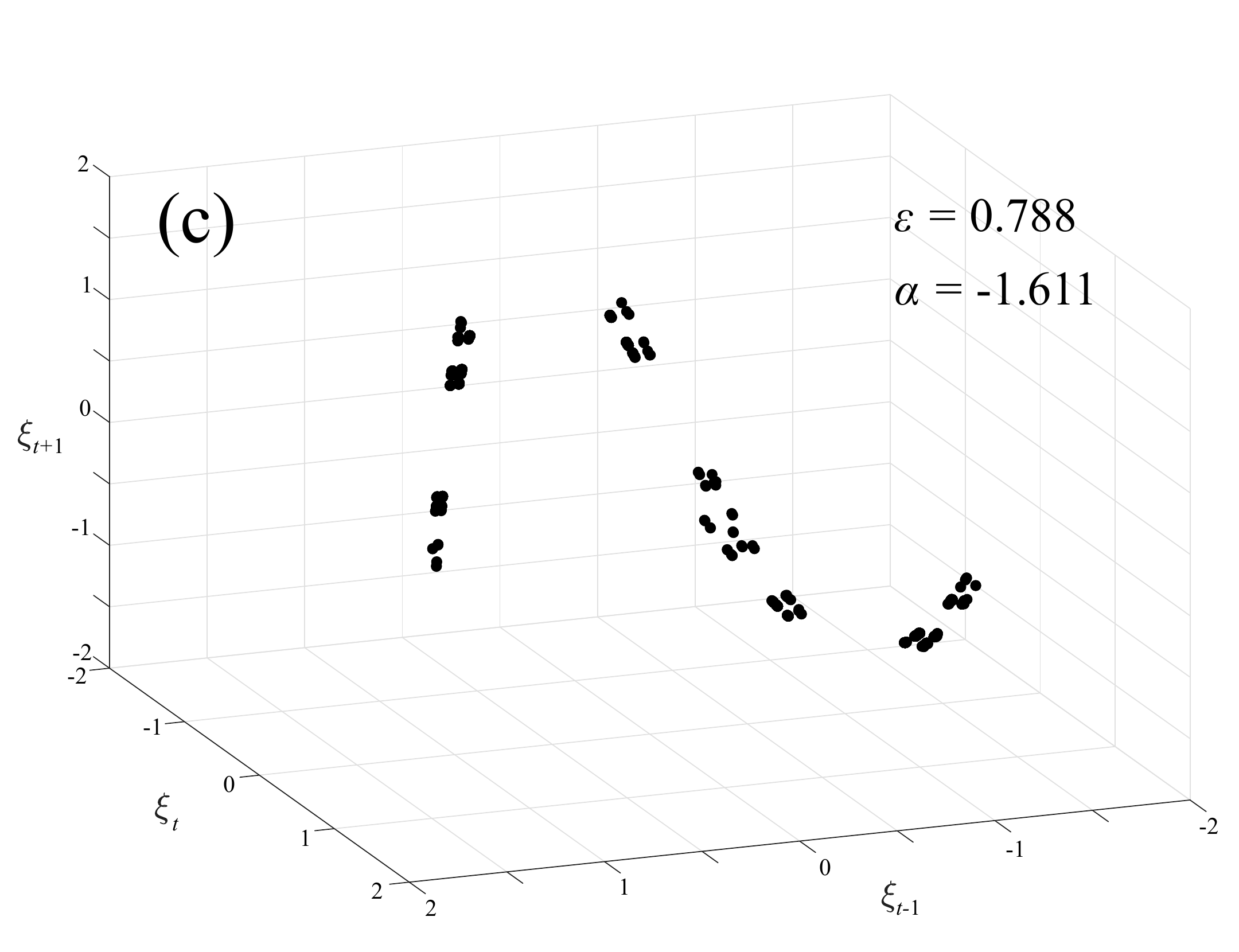}{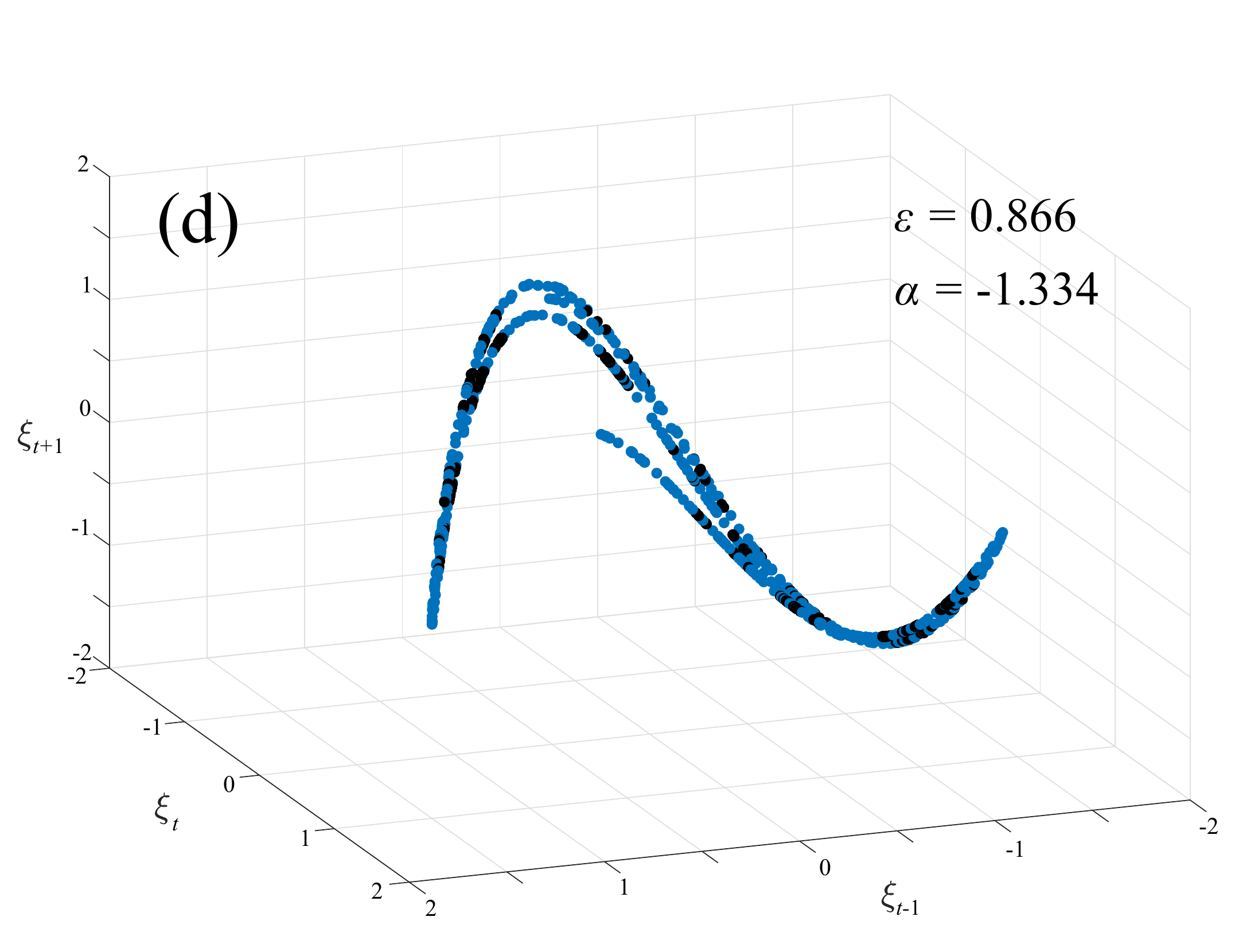}{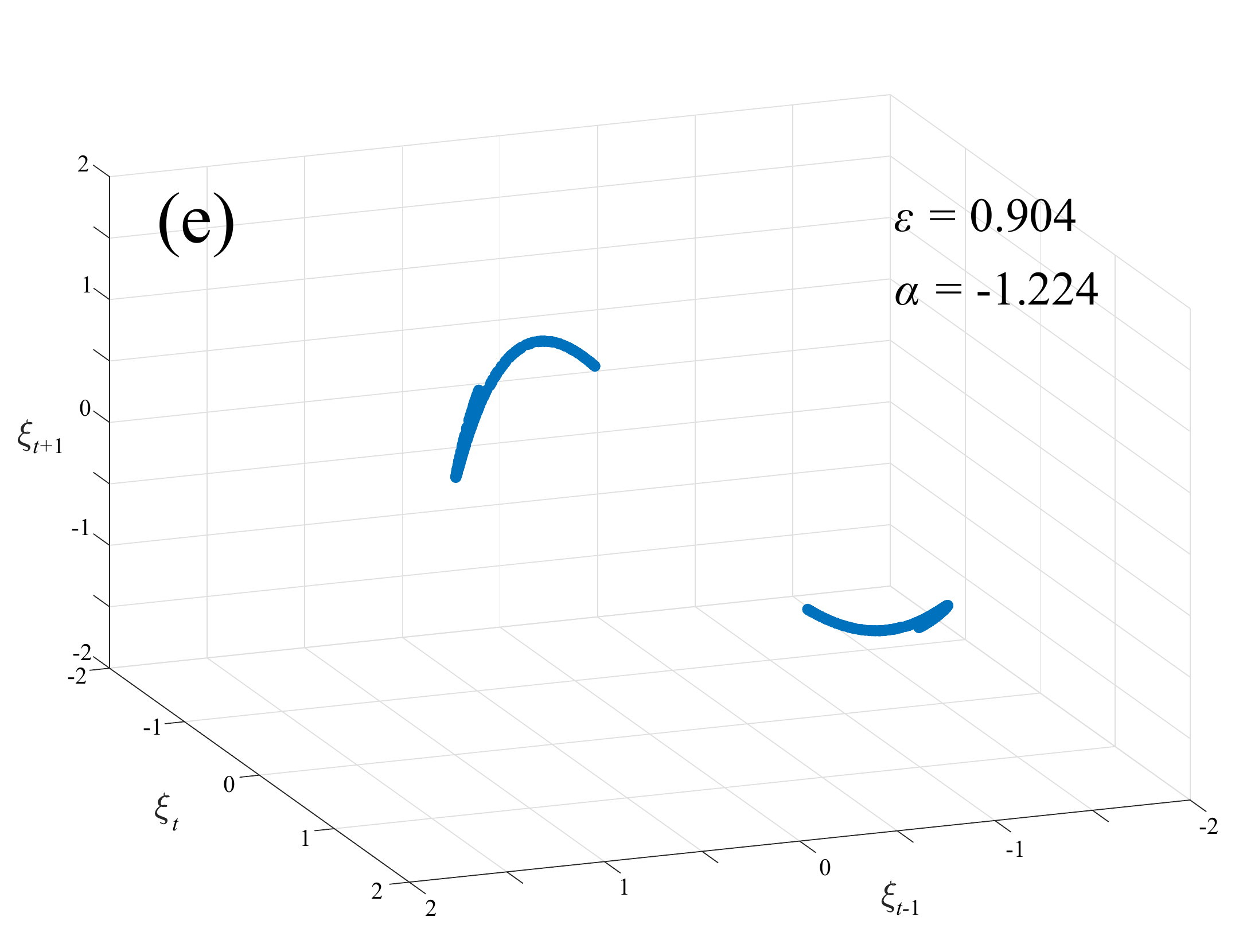}{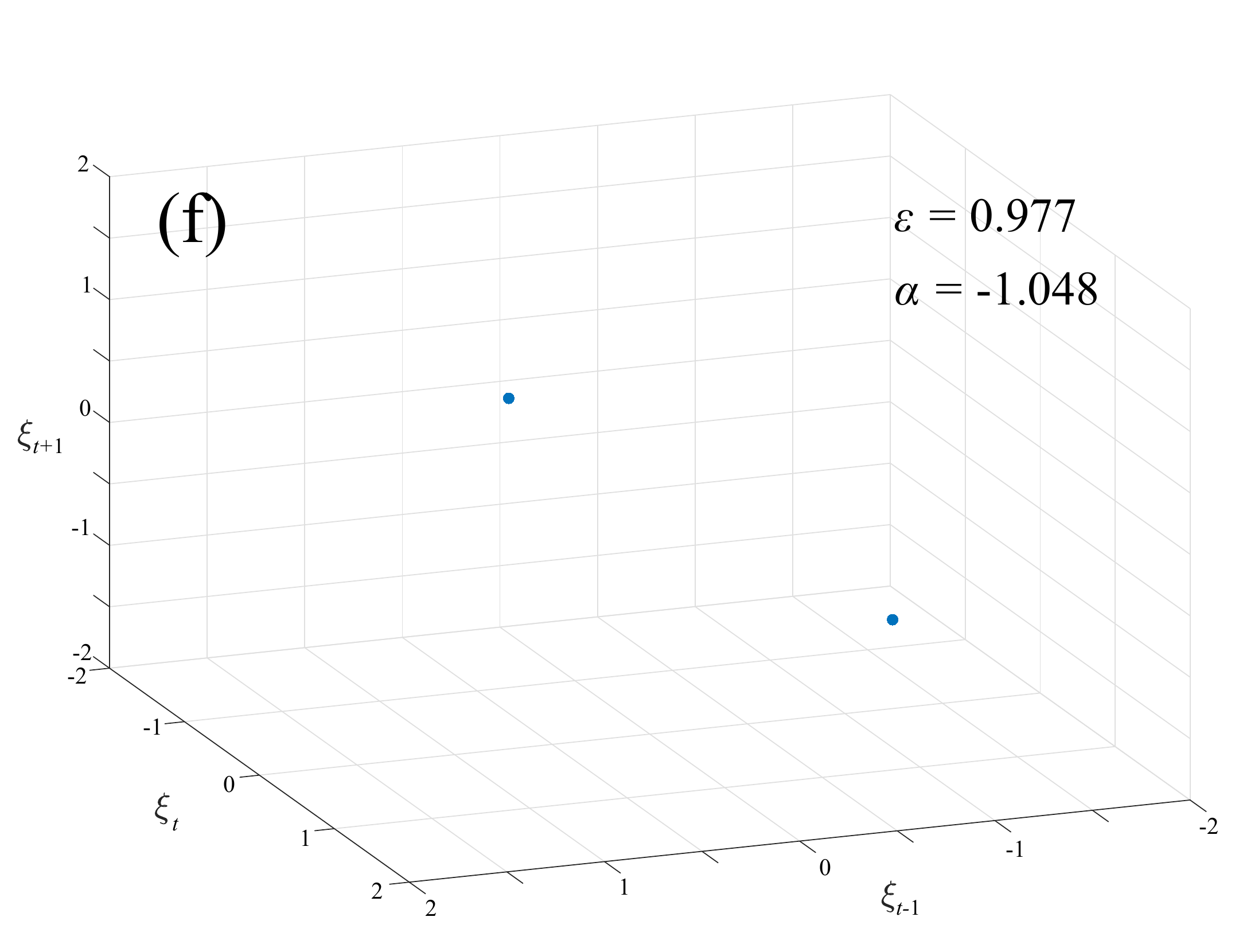}{The continuation of the period-$1200$ orbit for six values of $\eps$ (black) in $(\xi_{t-1},\xi_t,\xi_{t+1})$-space. Also shown are the attractors (blue) for the parameters with bounded orbits in \Fig{SCTongue}. 
(a,b,c) The periodic orbit is unstable, before it enters the region of bounded orbits. 
(d) For $\eps = 0.866$, the periodic orbit nearly coincides with a chaotic attractor. 
(e) Two chaotic bands that have (reverse) merged from the chaotic attractor seen in (d). (f) For larger $\eps$, the attractor goes through a (reverse) period-doubling cascade, two examples of which are shown here. 
A movie showing these results as $\eps$ varies can be seen at \href{https://drive.google.com/file/d/1THM2Cm8rUjzD56W6ShEtICK3njUpwOdv/view?usp=sharing}{here}.}{MoviePanels}{0.45}

%%%%%%%%%%
%%%%% Conclusion
%%%%%%%%%%%
\section{Conclusions}

In this paper we analyzed a two-parameter AI limit for the 3D quadratic diffeomorphism \Eq{Q3DMap},
extending our previous one-parameter results \cite{Hampton22}.
To obtain an AI limit, we now assume that both $\alpha$ and $\sigma$ tend to $\infty$, adding the ratio
$\sigma^2/\alpha =-r^2$ as an additional parameter; our previous results assumed that $\sigma$ remained finite so that only the single parameter $\alpha$ went to infinity.
In this new limit, the  AI states are still determined by a one-dimensional correspondence \Eq{ReducedAILimit},
but the added parameter $r$ shifts the center of the quadratic curve. 
To our knowledge, all previous studies of AI limits have been single parameter limits. 
There is still much to learn about how the different classes of AI limits---sending different
parameters to infinity---transition from one to the other as parameters vary.

In \Sec{Existence}, we generalized a result from \cite{Hampton22} to this case, obtaining a criterion for a one-to-one correspondence
between symbol sequences and AI state, \Lem{AIGeneralContraction}. Numerically computed parameter regions that satisfy the hypotheses of this lemma were found in \Sec{NumericalBounds}. We observed that these regions seem to converge onto a region $\cR_\cA$---simply defined by \Eq{MapsInto} and found in \Sec{AnalyticalBounds}---that \textit{can} be found analytically, recall \Eq{RA}.

In \Sec{SCcase} we used numerical continuation from an AI state to obtain orbits of \Eq{Q3DMap} 
as $\eps = 1/\sqrt{ -\alpha}$ grows from zero. 
For the case that we study, $b=c=0$ in \Eq{Q3DMap}, so that the resulting diffeomorphism has just one quadratic term---it can be thought of as a 3D version of H\'enon's quadratic map.  We chose the Jacobian of the map to be small, $\delta = 0.05$, so that the map is strongly volume-contracting. Thus this map is ``close'' to H\'enon's 2D map.

We studied the attractors of the map by looking at
the trajectory of an initial condition near a fixed point of \Eq{Q3DMap}.
When this trajectory remains bounded, it can limit to a periodic or chaotic attractor,
and these were classified in \Fig{SCTongue} over a range of $\alpha$ and $r$. 
We showed similar figures for related parameter scans in \cite{Hampton22b}. 
We observed that the low-period orbits that are stable in the ``Arnold-tongues'' of \Fig{SCTongue} correspond to orbits that are connected to those at the AI limit. We followed all the AI states up to period six,
finding their codimension-one bifurcations in \Sec{LowPeriod}.
These results are consistent with the, still unproven, ``no-bubbles'' conjecture for the 2D \hen map \cite{Sterling98}. Moreover, we observe that codimension-one bifurcations
occur between orbits whose AI sequences differ in exactly one symbol;
this is consistent with our previous results in \cite{Hampton22}.
Similarly, we conjecture that the symbol sequences of orbits arising from period-doubling
can be obtained from the parent orbit by doubling the sequence and then flipping exactly one symbol---the first, when ordered appropriately. 

A deficit of our construction of \Fig{SCTongue} is that we followed only the fate of a single intial condition. 
However the map \Eq{Q3DMap} can certainly have multiple attractors for fixed parameters.
We observed this, for example, for a pair of period-six orbits that we followed from the AI limit:
these continued into the period-doubling cascade of the fixed point, becoming stable in a parameter
domain where there is also a stable period-four orbit. This is also seen in \Fig{SCTongueDifferentIC}, which uses a different initial condition than that of \Fig{SCTongue}(a), and clearly exhibits the existence different periodic attractors.
One could speculate that continuation from the anti-integrable limit may be an efficient method to multiple
attractors and those with small basins of attraction, as well as to find unstable orbits. 
We hope to explore this approach to explore multiple attractors in future research.

Note that even though we only detected orbits up to period $80$ in \Fig{SCTongue}, we did not find 
parameter values for which there is an attracting periodic orbit with larger period.
This contrasts with some of the similar parameter scans
in our previous work \cite{Hampton22b}. We also do not observe attracting invariant circles for this
strongly contracting case---these were seen for larger $\delta$ in \cite{Hampton22, Hampton22b}.
In the future, we hope to further investigate the symbol sequences for AI states that evolve to invariant circles. 
In particular, it would be interesting to see if there is a similar relation between symbols 
and rotation numbers as that found in the 2D \hen case \cite{Dullin05}. Similarly, for the volume-preserving case, one could ask if there is a correspondence between symbols and the rotation vector for invariant tori.

%%%%%%%%%
%%%%Appendices
%%%%%%%%%%
\newpage
\appendix
%%%%%%%%%%
%%%%% Backwards Map
%%%%%%%%%%%
\section{Backwards Map}\label{app:BWMapCalcs}
Here we compute the region $\cR_\cA^-$ of \Sec{AnalyticalBounds}, imposing the condition \Eq{MapsInto} for the backwards AI map \Eq{BackwardsAIMap}. 

%This has the same fixed points $\xi_\pm$ \Eq{FixedPt} and derivative
%\[
%    g'(\xi_t)=\frac{-2 a s_t \xi_t}{\sqrt{r^2+4c-4ac\xi_t^2}}.
%\]
%\comment{Added $-s_t$, right?}

First suppose that $\Delta = -4ac = 0$. 
For the special case $c=0$, the backwards map reduces to $\xi_{t-1} = \tfrac{1}{r} (\xi_t^2-1)$, a simple, deterministic, 1D map: there is no nontrivial AI limit. Thus when $\Delta=0$, only the case $1-c = a=0$, where the curve \Eq{ReducedAILimit} is a pair of parallel lines,
\[
    \xi_{t-1} = -\tfrac12 (r +s_t \sqrt{r^2+4}),
\]
gives a nontrivial AI limit. In this case our arguments do apply using $B$ \Eq{BInterval} with $\beta = \xi_{max}$ \Eq{MaxFixPt}. The implication is that
\[
    \{(r,c):  c=1\} \subset \cR_\cA^{-},
\]
shown as the red line in \Fig{RC_PlaneAIExist}(c).

For the ellipse, $\Delta<0$, the map $g_\pm$ has domain given by the vertical bounds of the rectangle \Eq{EllipseBounds} and range given by its horizontal bounds. Thus to satisfy
$g_{s_t}(B) \subset B$, we set $\beta=\frac{1}{2c}(|r|+\sqrt{r^2+4c})$ and require
\[
  \beta <\sqrt{\frac{r^2+4c}{4ac}}.
\]
This parameter region in which this is satisfied is
\[
    \{(r,c): \cC_3(r)<c<1\},
\]
where $\cC_3$ the largest root of \Eq{EllipsePoly}. This is the tan colored region in \Fig{RC_PlaneAIExist}(c) and (d).

The hyperbolic case, $\Delta>0$ case requires a bit more work, as the backward $\pm$ maps no longer have the up-down reflection symmetry of the forward map. Nevertheless we can still take $B= [-\xi_{max},\xi_{max}]$.
In order that the maps $g_{\pm}$ are well-defined and give distinct orbits, the radicand of \Eq{BackwardsAIMap} must be positive,
\[
    r^2+4c-4ac\xi_t^2>0.
\]
When $c >1$, so that $a = 1-c < 0$, this is always true.
Note that when $c>1$, the asymptotes of the hyperbola \Eq{Asymptotes} have a slope less than one when thought of as $\xi_{t-1}$ as  function of $\xi_t$. Thus $g_{s_t}$ is a contraction on $B$. This implies that
\[
    \{(r,c): c>1 \} \subset \cR_\cA^{-},
\]
as shown in red in \Fig{RC_PlaneAIExist}(c) and (d).

%%%%%%%
%%%% n=2 Parabola
%%%%%%%
\section{Parabolic Case: Two Iterates}\label{app:ParabolicTwo}

For the case $n=2$, we require $\|D\cF^2\|_\infty <1$. This will necessarily give a larger parameter interval, $\cR^+_2 \supset \cR^+_1$, since  the product of two slopes can be less than one even when one of them is larger than one. The composition
% \[
    % \xi_t=f_{s_t}(\xi_{t-1})=s_t\sqrt{r\xi_{t-1}+1} \quad \text{ and } \xi_{t+1}=f_{s_{t+1}}(\xi_{t})=s_{t+1}\sqrt{r\xi_{t}+1}
% \]
\[
f_{s_{t+1}}(f_{s_t}(\xi_{t-1}))=s_{t+1}\sqrt{rs_t\sqrt{r\xi_{t-1}+1}+1},
\]
has derivative
\beq{2ndParallelDerivative}
f'_{s_{t+1}}(f_{s_t}(\xi_{t-1}))f_{s_t}'(\xi_{t-1})=\frac{s_t s_{t+1} r^2}{4 \sqrt{(rs_t\sqrt{r\xi_{t-1}+1}+1)(r\xi_{t-1}+1)}}.
\eeq
This has magnitude $1$ when  $\xi$ is a root of the cubic
%\[
%1 = \frac{r^2}{4 \sqrt{(rs_t\sqrt{r\xi+1}+1)(r\xi+1)}}.
    % r^2 &= 4 \sqrt{(rs_t\sqrt{r\xi+1}+1)(r\xi+1)} \\
    % r^4 &= 16(rs_t\sqrt{r\xi+1}+1)(r\xi+1) = 16(rs_t(r\xi+1)^{3/2}+(r\xi+1)) \\
    % r^4-16(r\xi+1) &= 16rs_t(r\xi+1)^{3/2}. 
%\]
%After some algebra, this leads to a cubic 
polynomial
%\[
    %(r^4-16(r\xi_{t-1}+1))^2 &= 256r^2(r\xi_{t-1}+1)^3 \\
    %256 - 32 r^4 + r^8 + 512 r \xi_{t-1} - 32 r^5 \xi_{t-1} + 256 r^2 \xi_{t-1}^2 = 256 r^2 + 768 r^3 \xi_{t-1} + 768 r^4 \xi_{t-1}^2 + 256 r^5 \xi_{t-1}^3
%\]
\[
%    P(x) = 256 r^5x^3+x^2(768 r^4 -256 r^2  )+ x(32 r^5+768 r^3-512 r)) - r^8+ 32 r^4+256 r^2- 256
    P_P(\xi_{t-1}) = 256 r^5\xi_{t-1}^3 +256 r^2 (3 r^2-1)\xi_{t-1}^2 +32 r(r^4+24 r^2-16)\xi_{t-1}  - r^8+ 32 r^4+256 r^2- 256
    %-(r^4+16 r -16) (r^4-16 r -16) 
\]
Note that the discriminant of this polynomial, $2^{16}r^{20}(64-27r^6)$, is always positive on the interval $|r|<\frac{1}{\sqrt{2}}$;
thus the three roots of $P_P$ are real. 
Thus to enforce the derivative \Eq{2ndParallelDerivative} to have magnitude less than one, the roots must lie outside $B$.
%of the roots in magnitude must be larger than the maximum fixed point
%\[
%   \min{(|\cC^{||}_1|,|\cC^{||}_2|,|\cC^{||}_3|)} > \beta
%\]
Numerically, this gives the bound
\[
    |r|\lesssim 0.6984177, 
\]
which is only a slight improvement over the case $n=1$.

%%%%%%%%%%
%%%%% Continuation
%%%%%%%%%%%
\section{Continuation Algorithm}\label{app:NumericalContinuation}

The persistence of AI states away from the AI limit can be proven using contraction arguments \cite{Sterling98,Hampton22}.
Straightforward numerical continuation can be used for periodic orbits of \Eq{Q3DMap}. 
The continuation algorithm is based on reformulating the  difference equation \Eq{RescaledDifEq}
to use a predictor-corrector method. A period-$n$ orbit, i.e., a sequence
\[
    \xi \in  \{\xi \in \bR^\infty \mid \xi_{t+n} \equiv \xi_{t}, \, \forall t \in \bZ \} \simeq \bR^n ,
\]
must be a zero of the function $\cG: \bR^n \times \bR \to \bR^n$ defined by
\[
   \cG(\xi,\eps) =\left( \cL_\eps(\xi_1,\xi_0,\xi_{n-1},\xi_{n-2}), \cL_\eps(\xi_2,\xi_1,\xi_{0},\xi_{n-1}), \ldots ,
    \cL_\eps(\xi_{0},\xi_{n-1},\xi_{n-2},\xi_{n-3}) \right) .
\]
We use a standard pseudo-arclength continuation algorithm \cite[Sec. 1.2.3]{Krauskopf07}, to find a discretization of a curve of solutions, $\cG(\xi,\eps) = 0$, at discrete points $(\xi^k,\eps^k)  \in \bR^n \times \bR$, for $k=0,1,\ldots$. 

Given a solution  $(\xi^k,\eps^k)$ at index $k$, a direction vector, $v^k = (\dot{\xi}^k,\dot{\eps}^k)$, and a predetermined arclength step size, $\ell$, the goal is to find a new solution in the hyperplane
orthogonal to $v^k$ at a distance $\ell$ from the previous solution.
Thus, to obtain $(\xi^{k+1}, \eps^{k+1})$, one must solve the system 
\bsplit{Arclength}
    &\cG(\xi^{k+1}, \eps^{k+1}) = 0 ,\\
    &\dot{\xi}^k(\xi^{k+1}-\xi^k)+\dot{\eps}^k(\eps^{k+1}-\eps^k)=\ell .
\esplit
This can be done iteratively, beginning with the point in the hyperplane along $v^k$, 
$(\xi^{k+1}, \eps^{k+1}) =  (\xi^k, \eps^k) + \frac{\ell}{\|v^k\|} v^k$, i.e.,
as an initial guess.
A solution to \Eq{Arclength} is then obtained using  Broyden's quasi-Newton method
to approximate the Jacobian of \Eq{Arclength} and 
a $QR$-decomposition to find its inverse \cite{Allgower90}. 
In our computations, the Broyden iteration stops when $\|\cG\|_\infty < 10^{-12}$
or after a predetermined maximum number of steps, here set to $150$.

The algorithm is initialized with an AI state, $(\xi^{0},0)$, which is found by iterating $f_{s_t}$ \Eq{AIMap}
for the given  period-$n$ symbol sequence $s$, beginning with a randomly selected point $ \xi_0 \in B$, and then iterating until
the orbit converges to a tolerance of $10^{-12}$. 
The initial vector $v^0 = (\dot{\xi}^0,0.005)$ is chosen so
that $\dot{\xi}^0$ solves the first $n$ rows of 
\beq{tangentVector}
    \begin{pmatrix}
        \partial_\xi \cG(\xi^k,\eps^k) & \partial_\eps \cG(\xi^k,\eps^k) \\
        \dot{\xi^k}^T & \dot{\eps}^k
    \end{pmatrix}
    \begin{pmatrix}
        \dot{\xi}^{k+1}\\
        \dot{\eps}^{k+1}
    \end{pmatrix}
    = \begin{pmatrix}
         0\\ 1
    \end{pmatrix} .
\eeq
Subsequently, each new direction vector $(\dot{\xi}^{k+1},\dot{\eps}^{k+1})$ is found by solving the full system \Eq{tangentVector} to obtain a normalized tangent vector.
 
In addition to the mapping parameters $a,c,r,$ and $\delta$, we choose an initial arclength step size $\ell$, which varies depending on the period $n$ and values of the parameters. For periodic orbits with length less than 10, we chose $\ell = 10^{-2}$. For longer orbits, we chose $\ell= 10^{-1}$. During the continuation process, $\ell$ is decreased by a factor of two if the solution is more than a distance of $0.1$ from the previous solution (i.e., if the solutions jumps ``too far away'').
The continuation runs until $\eps$ `turns around', i.e., $\eps^{k+1} < \eps^{k}$, or until $\ell$ becomes smaller than 
$10^{-15}$.

\bibliography{AIBibliography}
\bibliographystyle{alpha}

\end{document}